\documentclass[12pt]{amsart}
\usepackage{amsthm,amssymb}
\usepackage{hyperref}
\usepackage{xcolor}
\usepackage{verbatim}
\usepackage[cmtip,arrow]{xy}
\usepackage{pb-diagram,pb-xy}
\usepackage[margin=1.1in]{geometry}
\usepackage[colorinlistoftodos]{todonotes}

\newtheorem{thm}{Theorem}
\newtheorem{theorem}[thm]{Theorem}
\newtheorem{lem}[thm]{Lemma}

\newtheorem{prop}[thm]{Proposition}
\newtheorem{cor}[thm]{Corollary}

\theoremstyle{remark}

\newtheorem{rem}[thm]{Remark}
\newtheorem{remark}[thm]{Remark}

\numberwithin{thm}{section}
\numberwithin{equation}{section}

\newcommand{\af}{\mathrm{af}}
\newcommand{\al}{\alpha}

\newcommand{\flags}{\hat{\mathrm{Fl}}}

\newcommand{\agp}{\tilde{\mathfrak{G}}}
\newcommand{\asp}{\tilde{\mathfrak{S}}}
\newcommand{\bA}{{\hat{\mathbb{A}}}}

\newcommand{\bKfin}{\mathbb{K}}
\newcommand{\bC}{\mathbb{C}}
\newcommand{\bK}{{\hat{\mathbb{K}}}}

\newcommand{\bZ}{\mathbb{Z}}
\newcommand{\C}{\mathbb{C}}

\newcommand{\cl}{\mathrm{cl}}
\newcommand{\cL}{\mathcal{L}}
\newcommand{\cO}{\mathcal{O}}
\newcommand{\diag}{\mathrm{diag}}
\newcommand{\End}{\mathrm{End}}
\newcommand{\ev}{\mathrm{ev}}

\newcommand{\F}{{\rm Fun}}
\newcommand{\Frac}{{\rm Frac}}

\newcommand{\gr}{\thickaffgr}

\newcommand{\thickaffgr}{\mathrm{Gr}}

\newcommand{\Hom}{\mathrm{Hom}}
\newcommand{\Iaf}{\hat{I}}
\newcommand{\id}{\mathrm{id}}
\newcommand{\Ifin}{I}
\newcommand{\ip}[2]{\langle #1, #2 \rangle}

\newcommand{\LL}{{\mathbb L}}
\newcommand{\La}{\Lambda}
\newcommand{\la}{\lambda}
\newcommand{\level}{\mathrm{level}}
\newcommand{\LG}{{\mathrm {LG}}}
\newcommand{\loc}{\mathrm{loc}}
\newcommand{\OG}{{\mathrm{OG}}}
\newcommand{\PP}{{\mathbb P}}

\newcommand{\pair}[2]{\langle #1\,,\,#2\rangle}
\newcommand{\pt}{\mathrm{pt}}
\newcommand{\Q}{{\mathbb Q}}
\newcommand{\R}{{\mathbb R}}

\newcommand{\bS}{\overleftarrow{\mathfrak{S}}}
\newcommand{\schub}{\mathfrak{S}}
\newcommand{\SL}{{\mathrm{SL}}}
\newcommand{\SO}{{\mathrm{SO}}}
\newcommand{\Sp}{{\mathrm{Sp}}}

\newcommand{\Sym}{\mathrm{Sym}}
\newcommand{\Taf}{\hat{T}}
\newcommand{\tonilHecke}{\tau}
\newcommand{\tocentralizer}{k}
\newcommand{\TR}{{T_{\mathbb{R}}}}
\newcommand{\tS}{{\tilde S}}

\newcommand{\ty}{\tilde y}
\newcommand{\Wa}{\hat{W}}
\newcommand{\Wf}{W}
\newcommand{\Wz}{\Wa^0}
\newcommand{\Xfin}{P}
\newcommand{\Xaf}{\hat{P}}
\newcommand{\Z}{{\mathbb Z}}

\def\dnode#1#2{\overset{#1}{\underset{#2}{\circ}}}
\def\ver#1#2{\overset{{\llap{$\scriptstyle#1$}\displaystyle\circ{\rlap{$\scriptstyle#2$}}}}{\scriptstyle\vert}}

\begin{document}
\title{On the coproduct in affine Schubert calculus}
\author{Thomas Lam}
\address{Department of Mathematics \\
University of Michigan \\
530 Church St. \\
Ann Arbor 48109 USA}
\email{tfylam@umich.edu}
\thanks{T.L. was supported by NSF DMS-1464693 and DMS-1953852, and by a von Neumann Fellowship from the Institute for Advanced Study.}
\author{Seung Jin Lee}
\address{Department of Mathematical Sciences \\ Seoul National University \\ GwanAkRo 1 \\ Gwanak-Gu Seoul 08826 Korea}
\email{lsjin@snu.ac.kr}
\thanks{S. J. Lee was supported by the National Research Foundation of Korea(NRF) grant funded by the Korea government(MEST) (No. 2019R1C1C1003473).}
\author{Mark Shimozono}
\address{Department of Mathematics \\
460 McBryde Hall, Virginia Tech\\
 255 Stanger St. \\
Blacksburg, VA, 24601, USA }
\email{mshimo@math.vt.edu}
\thanks{M.S. was supported by NSF DMS-1600653.}

\dedicatory{Dedicated to Bill Fulton on the occasion of his 80th birthday.
\\ Thank you, Bill, for your inspirational and visionary work!
}

\begin{abstract}
The cohomology of the affine flag variety $\flags_G$ of a complex reductive group $G$ is a comodule over the cohomology of the affine Grassmannian $\gr_G$.
We give positive formulae for the coproduct of an affine Schubert class in terms of affine Stanley classes and finite Schubert classes, in (torus-equivariant) cohomology and $K$-theory.  As an application, we deduce monomial positivity for the affine Schubert polynomials of the second author.
\end{abstract}

\maketitle

\section{Introduction}
Let $G$ be a complex reductive group with maximal torus $T$ and flag variety $G/B$, and denote by $\xi^v_{G/B}$ the Schubert classes of $H^*_T(G/B)$ (all cohomology rings are taken with integer coefficients), indexed by the finite Weyl group $\Wf$.  Let $\flags_G$ denote the affine flag variety of $G$ and $\gr_G$ denote the affine Grassmannian of $G$.  There is a coaction map
$$
\Delta: H^*_T(\flags_G) \to H^*_T(\gr_G) \otimes_{H^*_T(\pt)} H^*_T(\flags_G).
$$
It is induced via pullback from the product map of topological spaces
$
\Omega K \times LK/T_\R \to LK/T_\R,
$
where $K \subset G$ is a maximal compact subgroup and $T_\R = K \cap T$ is the maximal compact torus.  The cohomology ring $H^*_T(\flags_G)$ has Schubert classes $\xi^w$ indexed by the affine Weyl group $\Wa$.  The inclusion $\varphi: \Omega K \hookrightarrow LK/T_\R$ induces a ``wrongway" pullback map $$\varphi^*: H^*_T(\flags_G) \to H^*_T(\gr_G).$$  By definition, the equivariant affine Stanley class $F^w \in H^*_T(\gr_G)$ is given by $F^w:= \varphi^*(\xi^w)$.  We refer the reader to \cite{book} for further background.
%We also have an evaluation at identity map $\ev_1: LK/T_\R \to K/T_\R$ inducing a pullback map 
%$$\ev_1^*:H^*_T(G/B) \to H^*_T(\flags_G).$$

\begin{theorem}\label{T:main}
Let $w \in \Wa$.  Then we have
$$
\Delta(\xi^w) = \sum_{w \doteq uv }F^u \otimes \xi^v
$$
and under the isomorphism $H^*_T(\flags_G) \cong H^*_T(\gr_G) \otimes_{H^*_T(\pt)} H^*_T(G/B)$, $$
\xi^w = \sum_{w \doteq uv } F^u \otimes  \xi^v_{G/B}
$$
where $u \in \Wa$ and $v \in \Wf$ and we write $w \doteq uv$ if $w = uv$ and $ \ell(w) = \ell(u) + \ell(v)$.
\end{theorem}
The class $\xi^v_{G/B}$ is considered an element of $H^*_T(\flags_G)$ via pullback under evaluation at the identity (see \eqref{eq:ev1}).  The same formulae hold in non-equivariant cohomology.  

 In the majority of this article (Sections~\ref{sec:nilHecke}--\ref{sec:Schubert}) we will work in torus-equivariant $K$-theory $K^*_T(\flags)$ of the affine flag variety.  The coproduct formula for holds in (torus-equivariant) $K$-theory with Demazure product replacing length-additive products (see Theorem \ref{T:affineschub}).  Our proof relies heavily on the action of the affine nilHecke ring on $K^*_T(\flags)$.  Let us note that there are a number of different geometric approaches \cite{KK,KS,LSS} for constructing Schubert classes in $K^*_T(\flags)$; see \cite[Section 3]{LLMS} for a comparison.  However, our results holds at the level of Grothendieck groups and the precise geometric model (thick affine flag variety, thin affine flag variety, or based loop group) is not crucial.
 
In Section~\ref{sec:cohom}, the proofs for the cohomology case are indicated.  

There is a long tradition of combinatorial formulae for Schubert classes in cohomology and $K$-theory using reduced factorizations or Hecke factorizations, dating at least back to \cite{LS:doubleSchub}, see also the references in Section~\ref{sec:examples}.  In particular, \cite{BJS} gives a formula for Schubert polynomials using reduced factorizations and \cite{Lam:affstan} gives a formula for affine Stanley symmetric functions using cyclically decreasing reduced factorizations.  In Section~\ref{sec:examples} we combine these formulae with our Theorem~\ref{T:main} to prove (Theorem \ref{T:monomialpositive}) that the affine Schubert polynomials \cite{Lee} are monomial positive.  We explain how the Billey--Haiman formula \cite{BH} for type $C$ or $D$ Schubert polynomials (see also \cite{IMN}) is a consequence of our coproduct formula.

By taking an appropriate limit (see Section~\ref{sec:examples}), 
the coproduct formula for backstable (double) Schubert polynomials \cite{LLS} can be deduced from Theorem \ref{T:main}.  Whereas the proofs in \cite{LLS} are essentially combinatorial, the present work relies heavily on equivariant localization and the nilHecke algebra.

\medskip
\noindent
{\bf Acknowledgements.} The authors thank an anonymous referee for comments and corrections.

\section{Affine nilHecke ring and the equivariant $K$-theory of the affine flag variety}\label{sec:nilHecke}
The proofs of our results for a complex reductive group easily reduces to that of a semisimple simply-connected group.  To stay close to our main references \cite{KK,LSS}, we work with the latter.  Henceforth, we fix a complex semisimple simply-connected group $G$.  

The results of this section are due to Kostant and Kumar \cite{KK}.  Our notation %agrees with that of \cite{LSS}.
follows that of \cite{LSS}.
%\subsection{Weight lattices}
%\label{SS:weightlattices}

%Let $\af:\Xfin\to\Xaf$ be the section of $\cl$ given by $\af(\omega_i) = \La_i - \level(\La_i)\La_0$ for $i\in\Ifin$, where $\level(\La)$ is the level of $\La\in\Xaf$ \cite{Kac}.

\subsection{Small-torus affine $K$-nilHecke ring}
\label{S:nilHecke}
Let $T \subset G$ be the maximal torus with character group, or weight lattice $\Xfin$.  We have $\Xfin = \bigoplus_{i\in \Ifin} \bZ \omega_i$ where $\omega_i$ denotes a fundamental weight and $\Ifin$ denotes the finite Dynkin diagram of $G$.  Let 
$\Xaf = \bZ\delta \oplus \bigoplus_{i\in\Iaf} \bZ\La_i$ be the affine weight lattice with fundamental weights $\La_i$ for $i$ in the affine Dynkin node set $\Iaf=\Ifin\cup\{0\}$, and let $\delta$ denote the null root. Let $\Xaf^* =\Hom_{\bZ}(\Xaf,\bZ)$ be the affine coweight lattice. It has basis dual to $\{\delta\}\cup \{\La_i\mid i\in \Iaf\}$ given by $\{d\} \cup \{\alpha_i^\vee\mid i\in \Iaf\}$ where $d$ is the degree generator and the $\alpha_i^\vee$ are the simple coroots. The Cartan matrix $ (a_{ij}\mid i,j\in \Iaf)$ is defined by $a_{ij} = \pair{\alpha_i^\vee}{\alpha_j}$ using the evaluation pairing $\Xaf^*\times\Xaf \to \bZ$.
Let $c\in \Xaf^*$ be the canonical central element \cite[\S6.2]{Kac}. The level of $\La\in \Xaf$ is defined by $\level(\La)=\pair{c}{\La}$. The natural projection $\cl:\Xaf\to\Xfin$ has kernel $\bZ\delta\oplus \bZ \La_0$ and satisfies $\cl(\La_i)=\omega_i$ for $i\in\Ifin$. In particular $\cl(\alpha_0)=-\theta$ where $\theta$ is the highest root. This induces a map $\cl:\Z[\hat{P}]\to\Z[P]$ between the representation ring of the maximal torus of the affine Kac-Moody group and that of the torus $T$.
Let $\af:\Xfin\to\Xaf$ be the section of $\cl$ given by $\af(\omega_i) = \La_i - \level(\La_i)\La_0$ for $i\in\Ifin$.
%, where $\level(\La)$ is the level of $\La\in\Xaf$ \cite{Kac}.

The finite Weyl group $\Wf$ acts naturally on $\Xfin$ and on $R(T)$, where $R(T)\cong\Z[\Xfin]=\bigoplus_{\la\in \Xfin} \Z e^\la$ is the Grothendieck group of the category of finite-dimensional
$T$-modules, and for $\la\in P$, $e^\la$ is the class of the
one-dimensional $T$-module with character $\la$. Let $Q(T) = \Frac(R(T))$.
The affine Weyl group $\Wa$ also acts on $\Xfin$, $R(T)$, and $Q(T)$ via the level-zero action, that is, via the homomorphism $\cl_{\Wa}:\Wa\cong Q^\vee \rtimes \Wf \to \Wf$ given by $t_\mu v\mapsto v$ for $\mu$ in the coroot lattice $Q^\vee$ and $v\in \Wf$. In particular, $s_0=t_{\theta^\vee}s_\theta$ satisfies $\cl(s_0)=s_\theta$ where %$\theta$ is the highest root and 
$\theta^\vee$ is %its associated coroot. 
the coroot associated to $\theta$.
%For $w \in \Wa$ and $f \in R(T)$, the action of a Weyl group element $w$ on $f$ is sometimes denoted by $\up{w}f$. Let $\Wz$ denote the minimal-length coset representatives in $\Wa/\Wf$.

We let $u*v \in \Wa$ denote the Demazure (or $0$-Hecke) product of $u,v \in \Wa$.  It is the associative product determined by 
\begin{align*}
s_i * v &= \begin{cases} s_i v & \mbox{if $s_i v > v$,} \\
v & \mbox{otherwise.}
\end{cases} &
 v*s_i &= \begin{cases} vs_i & \mbox{if $vs_i > v$,} \\
v & \mbox{otherwise.}
\end{cases} 
\end{align*}

Let $\bK_{Q(T)}$ be the smash product of the group algebra $\Q[\Wa]$
and $Q(T)$, defined by $\bK_{Q(T)} = Q(T) \otimes_\Q \Q[\Wa]$
with multiplication $$(q\otimes w)(p\otimes v) = q (w\cdot p) \otimes wv$$
for $p,q\in Q(T)$ and $v,w\in W$. We write $qw$ instead of $q\otimes w$.
%For $i\in \Iaf$ define the Demazure operator $y_i\in \bK_{Q(T)}$ by
%\begin{equation*}
%  y_i = (1-e^{-\al_i})^{-1} (1 - e^{-\al_i} r_i).
%\end{equation*}
%The $y_i$ are idempotent and satisfy the braid relations:
%\begin{equation*}
% y_i^2 = y_i  \ \ \ \ \text{and}\ \ \ \
%\underbrace{y_iy_j\dotsm}_{\text{$m_{ij}$ times}} =
%\underbrace{y_jy_i\dotsm}_{\text{$m_{ij}$ times}}\; .
%\end{equation*}
Define the elements $T_i\in \bK_{Q(T)}$ by
\begin{equation} \label{E:Tdef}
  T_i %= y_i - 1 
  = (1-e^{\al_i})^{-1} (s_i-1).
\end{equation}
In particular $T_0 = (1-e^{-\theta})^{-1}(t_{\theta^\vee} s_\theta - 1)$.
%We have
%\begin{equation} \label{E:r}
%  r_i = 1 + (1-e^{\al_i}) T_i.
%\end{equation}
The $T_i$ satisfy
\begin{equation} \label{E:Tbraid}
T_i^2 = -T_i \ \ \ \ \text{and}\ \ \ \
\underbrace{T_iT_j\dotsm}_{\text{$m_{ij}$ factors}} =
\underbrace{T_jT_i\dotsm}_{\text{$m_{ij}$ factors}}
\end{equation}
where $m_{ij}$ is related to the Cartan matrix entries $a_{ij}$ by
\begin{align*}
\begin{array}{|c||c|c|c|c|c|}\hline \hline
a_{ij}a_{ji} & 0 & 1 & 2 & 3 & \ge 4 \\  \hline
m_{ij} & 2 & 3&4&6&\infty\\ \hline
\end{array}
\end{align*}
We have the commutation relation in $\bK_{Q(T)}$
\begin{equation}\label{E:Tcomm}
T_i \,q = (T_i\cdot q) + (s_i \cdot q) T_i\qquad\text{for $q\in Q(T)$.}
\end{equation}
Let $T_w = T_{i_1} T_{i_2}\dotsm T_{i_N}\in \bK_{Q(T)}$
where $w=s_{i_1}s_{i_2}\dotsm s_{i_N}$ is a reduced decomposition; it is well-defined  
by \eqref{E:Tbraid}. It is easily verified that
\begin{equation*}
  T_i T_w = \begin{cases}
  T_{s_i w} &\text{if $s_iw>w$} \\
  -T_w &\text{if $s_iw<w$}
  \end{cases}\qquad \text{and} \qquad
  T_w T_i = \begin{cases}
  T_{ws_i} &\text{if $ws_i>w$} \\
  -T_w &\text{if $ws_i<w$}
  \end{cases}
\end{equation*}
where $<$ denotes the Bruhat order on $\Wa$.
%For $\al\in\Phi^{+\re}$ define $T_\al = (1-e^\al)^{-1}(r_\al-1)$.
% Let $w\in \Wa$ and $i\in \Iaf$ be such that $\al=w\al_i$. Then
%\begin{equation} \label{E:Talconj}
%  T_\al = w T_i w^{-1}.
%\end{equation}
The algebra $\bK_{Q(T)}$ acts naturally on $Q(T)$.  In particular, one has
\begin{equation}\label{E:deriv}
  T_i \cdot (qq') = (T_i\cdot q) q' + (s_i\cdot q) T_i \cdot q'\qquad\text{for $q,q\in Q(T)$.}
\end{equation}
%\subsection{$0$-Hecke ring and integral form}
The $0$-Hecke ring $\bK_0$ is the subring of $\bK_{Q(T)}$ generated by
the $T_i$ over $\Z$. It can also be defined by generators $\{T_i\mid i\in \Iaf\}$
and relations \eqref{E:Tbraid}. We have $\bK_0 = \bigoplus_{w\in \Wa}
\bZ T_w$.

\begin{lem} \label{L:0Heckeacts} The ring $\bK_0$ acts on $R(T)$.
\end{lem}
\begin{proof} $\bK_0$ acts on $Q(T)$, and $T_i$ preserves $R(T)$ by \eqref{E:deriv}
and the following formulae for $\la\in P$:
\begin{equation} \label{eq:TonE}
  T_i \cdot e^\la = \begin{cases}
 % e^\la(1+e^{\al_i}+\dotsm+e^{(\ip{\al_i^\vee}{\la}-1)\al_i}) & \text{if $\ip{\al_i^\vee}{\la} > 0$} \\
 e^\la(e^{-\al_i}+e^{-2\al_i}\dotsm+e^{-\ip{\al_i^\vee}{\la}\al_i}) & \text{if $\ip{\al_i^\vee}{\la} > 0$} \\
  0 & \text{if $\ip{\al_i^\vee}{\la}=0$}\\
-e^\la(1+e^{\al_i}+\dotsm+e^{(-\ip{\al_i^\vee}{\la}-1)\al_i}) & \text{if $\ip{\al_i^\vee}{\la} < 0$.} 
  \end{cases}
  \qedhere
\end{equation}
\end{proof}

Define the {\it $K$-NilHecke ring} $\bK$ to be the subring of
$\bK_{Q(T)}$ generated by $\bK_0$ and $R(T)$.  We have $\bK_{Q(T)}
\cong Q(T) \otimes_{R(T)} \bK$.  By \eqref{E:Tcomm}, we have
\begin{equation} \label{E:Tbasis}
 \bK = \bigoplus_{w\in \Wa} R(T) T_w.
\end{equation}
%
%\subsection{Duality and function ``basis"}
%\fixit{Some overlap}
%Let $\F(\Wa,Q(T))$ be the right $Q(T)$-algebra of functions from $\Wa$ to $Q(T)$
%under pointwise multiplication and scalar multiplication $(\psi \cdot q)(w)=q \psi(w)$
%for $q\in Q(T)$, $\psi\in \F(\Wa,Q(T))$, and $w\in \Wa$.
%% stuff added
%By linearity, we identify $\F(\Wa,Q(T))$ with left $Q(T)$-linear
%maps $\bK_{Q(T)} \to Q(T)$:
%\begin{align*}
%  \psi(\sum_{w\in \Wa} a_w w) = \sum_{w\in \Wa} a_w \psi(w).
%\end{align*}
%$\F(\Wa,Q(T))$ is a $\bK_{Q(T)}-Q(T)$-bimodule via
%\begin{align}\label{E:Fbi}
%  (a\cdot \psi\cdot q)(b) = \psi(q b a) = q \psi(ba)
%\end{align}
%for $\psi\in \F(\Wa,Q(T))$, $q\in Q(T)$ and $a,b\in \bK_{Q(T)}$.
%
%
%Evaluation gives a perfect pairing $\ip{\cdot}{\cdot}: \bK_{Q(T)} \times
%\F(\Wa,Q(T)) \longrightarrow Q(T)$ given by
%\begin{equation*}
%   \ip{a}{\psi} = \psi(a).
%\end{equation*}
%It is $Q(T)$-bilinear in the sense that
%\begin{align*}
%  \ip{q a}{\psi} = q \ip{a}{\psi} = \ip{a}{\psi\cdot q}
%\end{align*}
%
%Define the subring $\Psi\subset\F(\Wa,Q(T))$ by
%\begin{align}
%\label{E:Psidef}
%  \Psi &= \{ \psi\in \F(\Wa,Q(T))\mid \psi(\bK) \subset R(T)\} \\
%\notag
%  &= \{ \psi\in \F(\Wa,Q(T))\mid \psi(T_w)\in R(T)\text{ for all $w\in \Wa$}\}.
%\end{align}
%
%Clearly $\Psi$ is a $\bK-R(T)$-bimodule.  By \eqref{E:Tbasis}, for
%$v\in W$, there are unique elements $\psi^v\in \Psi$ such that
%\begin{equation} \label{E:psidef}
%  \psi^v(T_w) = \delta_{v,w}.
%\end{equation}
%for all $w\in \Wa$.  We have $\Psi = \prod_{v\in \Wa} R(T) \psi^v$.

%
%\subsection{Localization for equivariant $K$-theory of affine flag variety}

\subsection{$\bK$-$\bK$-bimodule structure on equivariant $K$-theory of affine flag variety} 

We have an isomorphism $K^*_T(\pt) \cong R(T)$.  %We have an $R(T)$-algebra injection $\prod_{w \in \Wa} \iota^*_w : K^*_T(\flags) \hookrightarrow \prod_{w \in \Wa} K^*_T(\pt) \cong \F(\Wa,R(T))$, where $\F(\Wa,R(T))$ denotes the space of $R(T)$ valued functions on $R(T)$.  For $\psi \in K^*_T(\flags)$, write $\psi(w) := \iota^*_w(\psi)$, so that $\psi$ is identified with an element of $ \F(\Wa,R(T))$.  The subalgebra of $R(T)$-valued functions on $\Wa$ that can be obtained this way is characterized by the small torus affine GKM condition \cite{LSS}.  We denote the product on $K^*_T(\flags)$ by $\cup$.  Under the injection $K^*_T(\flags) \hookrightarrow \F(\Wa,R(T))$, it becomes the pointwise product on $\F(\Wa(R(T)))$.
Let $\F(\Wa,R(T))$ be the $R(T)$-algebra of functions $\Wa\to R(T)$ under pointwise multiplication $(\phi\psi)(w)=\phi(w)\psi(w)$ for $\phi,\psi\in\F(\Wa,R(T))$ and $w\in\Wa$,
and action $(s \psi)(w) = s \psi(w)$ for $s\in R(T)$, $\psi\in \F(\Wa,R(T))$ and $w\in \Wa$. There is an injective $R(T)$-algebra homomorphism $\mathrm{loc}: K_T^*(\flags) \to \F(\Wa,R(T))$ sending a class $\psi$ to the function $w\mapsto \psi(w)$ where $\psi(w)$ denotes the localization of $\psi$ at $w\in\Wa$. The image of the map $\loc$ is characterized by the small torus affine GKM condition of \cite[Section 4.2]{LSS}.

%There is a left $R(T)$-bilinear perfect pairing $\langle \cdot,\cdot \rangle: \bK \times %K^*_T(\flags)\to R(T)$ characterized by
There is a perfect left $Q(T)$-bilinear pairing $\pair{\cdot}{\cdot}: \bK_{Q(T)} \times \F(\Wa,Q(T)) \to Q(T)$ defined by evaluation:
\begin{align}\label{eq:pair}
\langle w, \psi \rangle = \psi(w)
\end{align}
for $w \in \Wa$ and %$\psi \in K^*_T(\flags)$.
$\psi\in\F(\Wa,Q(T))$.
Abusing notation, %for $a=\sum_{w\in\Wa} a_w w \in \K$ 
we regard every $\psi\in \F(\Wa,Q(T))$ as an element of
$\Hom_{Q(T)}(\bK_{Q(T)},Q(T))$ by formal left $Q(T)$-linearity: for
$a=\sum_{w\in\Wa} a_w w \in \bK_{Q(T)}$ 
with $a_w\in Q(T)$, let
\begin{align*}
	\psi(a) &= \pair{a}{\psi}= \sum_w a_w \psi(w).
\end{align*}
Thinking of $K^*_T(\flags)$ as %a subalgebra of $\Hom_S(\bK_{Q(T)},R(T))$
an $R(T)$-subalgebra of $\F(\Wa,Q(T))$, 
 a function $\psi$ lies in $K^*_T(\flags)$ if and only if $\psi(\bK) \subseteq R(T)$.
The pairing \eqref{eq:pair} restricts to a perfect left $R(T)$-bilinear pairing (see  \cite[(2.10)]{LSS})
\begin{align}\label{eq:integral pairing}
\bK \times K_T^*(\flags) \to R(T).
\end{align}

There is a left action $\psi \mapsto a \cdot \psi$ of $\bK$ on $K^*_T(\flags)$ given by the formulae (see \cite[Chapter 4, Proposition 3.16]{book} for the very similar cohomology case)
%\Thomas{I can't find where to cite}
%\Seungjin{I think this is essentially introduced in KK, although there are sign/ notation issues. Compare (4.7) in this paper and Proposition 2.22 (d) in KK}
%\Thomas{I don't think that's the same action.  We have two left actions, $\cdot$ and $\bullet$. As far as I recall, KK only has one, and I think it is the $\bullet$ one, but I might have remembered incorrectly.}
\begin{align}
(q \cdot \psi)(b) &= q\,\psi(b) \\
\label{eq:Tonfunc}
(T_i \cdot \psi)(b) &= T_i \cdot \psi(s_i b) + \psi(T_i b) \\
(w \cdot \psi)(b) &= w\,\psi(w^{-1} b)
\end{align}
for $b \in \bK$, $\psi \in K^*_T(\flags)$, $q \in R(T)$, $i \in \Iaf$, and $w \in \Wa$. Here, $T_i$ acts on $R(T)$ as in \eqref{E:deriv} and \eqref{eq:TonE}.

There is another left action $\psi \mapsto a \bullet \psi$ of $\bK$ on $K^*_T(\flags)$ given by \cite[\S 2.4]{LSS}
\begin{align}
(a \bullet \psi)(b)  = \psi(ba)
\end{align}
for $a,b \in \bK$ and $\psi \in K^*_T(\flags)$.

\begin{rem} For those familiar with the double Schubert polynomial
$\mathfrak{S}_w(x;a)$ (or also the double Grothendieck polynomial), the $\cdot$ action is on the equivariant variables $a_i$ and the $\bullet$ action is on the $x_i$ variables.
\end{rem}

Let $p: \flags \to \gr$ be the natural projection and
$p^*:K_T^*(\gr)\to K_T^*(\flags)$ the pullback map, which is an injection.  A class $\psi \in K_T^*(\flags)$ lies in the image of $p^*$ if and only if $\psi(wv) = \psi(w)$ for all $w \in \Wa$ and $v \in \Wf$.  We abuse notation by frequently identifying a class $\psi_\gr \in K_T^*(\gr)$ with its image under $p^*$.

%\subsection{Schubert basis and line bundles}
%The Schubert basis $\{\psi^v \mid v \in \Wa\} \subset K^*_T(\flags)$ is uniquely characterized by
%\begin{equation} \label{E:psidef}
% \psi^v(T_w) = \delta_{v,w}.
%\end{equation}

Let $[\cL_\la] \in K^*_T(\flags)$ denote the class of the $T$-equivariant line bundle %with weight $\la \in \Xfin$ on $\flags$.  
on $\flags$ of weight $\la$.
Using the level zero action of $\Wa$ on $R(T)$ we have \cite[(2.5)]{KS}
\begin{align}\label{E:linebundleloc}
	\pair{t_\mu v}{[\cL_\la]} &= v \cdot e^\la = e^{v\la} &\text {$\mu\in Q^\vee$, $v\in \Wf$.}
\end{align}

\begin{lem} \label{L:Xbullet} For any $\la\in \Xfin$ and $\psi\in K^*_T(\flags)$,
\begin{align}\label{E:Xbullet}
	e^\la \bullet \psi = [\cL_\la] \cup \psi.
\end{align}
\end{lem}
\begin{proof} Localizing at $t_\mu v$ for $\mu\in Q^\vee$ and $v\in\Wf$, we compute that $\pair{t_\mu v}{e^\la\bullet \psi}$ is equal to 
\begin{align*}
	%\pair{t_\mu v}{e^\la\bullet \psi}
	%= 
	\pair{t_\mu v e^\la}{\psi} 
	= \pair{e^{v\la} t_\mu v}{\psi} 
	= e^{v\la} \pair{t_\mu v}{\psi} 
	= \pair{t_\mu v}{[\cL_\la]} \pair{t_\mu v}{\psi} 
	= \pair{t_\mu v}{[\cL_\la] \cup \psi}. &&& \qedhere
\end{align*}
\end{proof}

\section{Endomorphisms of $K^*_T(\flags)$}\label{sec:endo}

\subsection{Wrong-way map and Peterson subalgebra}
Recall that $K \subset G$ is the maximal compact subgroup and $\TR:= K \cap T$ is the maximal compact torus. 
We have $\TR$-equivariant homotopy equivalences between $G/B$ and $K/\TR$, between $\gr$ and the based loop group $\Omega K$, and between $\flags$ and the space $LK/\TR$  \cite{M}.  For an ind-variety $X$ with $T$-action let $K_*^T(X)$ be the $T$-equivariant $K$-homology of $X$, the Grothendieck group of finitely supported $T$-equivariant coherent sheaves on $X$ \cite{Ku} \cite{LSS}. 
%\MarkS{We need the ind affine Grassmannian and affine flag ind-variety here. Why in $K$-theory can we just blithely go between the skinny and fat loop objects?}
%\Thomas{I don't think we can.  Instead the argument is we separately compute $K$-theory of skinny and fat objects and observe that the answer is the same.  Maybe we should say this?}
There is a left $R(T)$-module isomorphism $\tonilHecke: K_*^T(\flags)\cong \bK$ given by $\tonilHecke(\psi_w) = T_w$, where $\psi_w$ is the ideal sheaf Schubert class for the affine flag ind-variety (see Section \ref{ssec:Schubert}).  We give $K_*^T(\flags)$ the structure of a noncommutative ring so that $\tonilHecke$ is a ring isomorphism.  This ring structure can also be obtained geometrically from convolution; see \cite{Gi} for the corresponding statements for $H_*(G/B)$.  
The $K$-group $K_*^T(\gr)$ has the structure of a commutative Hopf $R(T)$-algebra. The product is induced from the $\TR$-equivariant product map of the topological group $\Omega K$.

%
%\MarkS{Make the Schubert class section precise. Give reference.}
%\Thomas{I added  the following sentence later: For the precise geometric interpretations of $\psi^x$ and $\psi_x$ we refer the reader to \cite[\S3]{LLMS}.}

There is a $\TR$-equivariant map $\varphi: \Omega K \to LK \to LK/\TR$ given by inclusion followed by projection. The map $\varphi$ induces an injective ring and left $R(T)$-module homomorphism $\varphi_*: K^T_*(\gr) \to K^T_*(\flags)$.
It also induces an $R(T)$-algebra homomorphism $\varphi^*:K_T^*(\flags)\to K_T^*(\gr)$ which is called the {\it wrong-way map}, and characterized by (see Lemma \ref{L:wrongwayloc})
\begin{align*}
\varphi^*(\psi)(v) &= \psi(t_\mu) &\mbox{for $v \in \Wz$ and $t_\mu \in v\Wf$.}
\end{align*}

 %They have Schubert basis $\{\psi^\gr_w \mid w \in \Wz\} \subset K_*^T(\gr)$ and $\{\psi_w \mid w \in \Wa\} \subset K_*^T(\flags)$.  

Let $\LL = Z_{\bK}(R(T))$ be the centralizer of $R(T)$ in $\bK$, called the $K$-Peterson subalgebra.  We have the following basic result \cite[Lemma 5.2]{LSS}.

\begin{lem}We have $\LL = \left(\bigoplus_{\mu \in Q^\vee} Q(T) t_\mu\right) \cap \bK$.
\end{lem}

\begin{thm}[{\cite[Theorem 5.3]{LSS}}] \label{T:centralizer}
There is an isomorphism $\tocentralizer: K_*^T(\gr) \to \LL$ making the following commutative diagram of ring and left $R(T)$-module homomorphisms:
%, where $\tonilHecke$ sends the homology Schubert class $\psi_w$ to $T_w$ for $w\in \Wa$.
%, and $\varphi_*^{\bK}$ is the inclusion map.
\begin{align*}
\begin{diagram}
\node{K_*^T(\gr)}\arrow {s,t}{\varphi_*} \arrow {e,t}{\tocentralizer} \node{\LL}  
\arrow {s,J}{}
\\ \node{K_*^T(\flags)} \arrow {e,b}{\tonilHecke} \node{\bK} 
\end{diagram}
\end{align*}
\end{thm}

\subsection{Pullback from affine Grassmannian}
Recall that $p: \flags \to \gr$ denotes the natural projection. 
%Note that for all $u\in\Wz$ and $v\in\Wf$, $p$ sends the $T$-fixed point $uv$ of $\flags$ to the $T$-fixed point $u$ of $\gr$. Thus we have
%\begin{align}
%\label{E:ppullback}
%  p^*(\psi)(uv) &= \psi(u) \qquad\text {for all $u\in\Wz$, $v\in\Wf$, $\psi\in K_T^*(\thickaffgr)$.}
%\end{align}
Define $\theta:= p^* \circ \varphi^*$, so that $\theta: K^*_T(\flags) \to K^*_T(\flags)$ is the pullback map in equivariant $K$-theory of the following composition
\begin{align}\label{eq:theta}
LK/\TR \xrightarrow{\;\;\;p\;\;\;} \Omega K \xrightarrow{\;\;\;\varphi\;\;\;} LK/\TR
\end{align}
where abusing notation, we are denoting also by $p$ the natural quotient map $LK/\TR \to LK/K \simeq \Omega K$.

\begin{lem} \label{L:wrongwayloc}
	For all $\mu\in Q^\vee$, $v\in\Wf$, and $\psi \in K^*_T(\flags)$ we have	\begin{align}
		(\theta \psi)({t_\mu v}) = \psi(t_\mu).
	\end{align}
\end{lem}
\begin{proof} The translation element $t_\mu$ defines the based loop given by the cocharacter
$\mu\in \Hom_{\text{alg. gp}}(\C^*,T)$ evaluated on the unit circle $S^1 \subset \C^*$. The unique $\TR$-fixed point in $t_\mu v  K \cap \Omega K$ is 
$t_\mu$. Thus under the composition \eqref{eq:theta} $t_\mu v$ maps to $t_\mu$. The Lemma follows by the definition of pullback.
\end{proof}

\subsection{Coaction}

The inclusion $\Omega K\hookrightarrow LK$ induces an action $\Omega K \times LK/\TR \to LK/\TR$ of $\Omega K$ on $LK/\TR$. This action is $\TR$-equivariant where $\TR$ acts diagonally on the direct product, acting on $\Omega K$ by conjugation and on $LK/\TR$ by left translation. Applying the covariant functor $K_*^\TR$
we obtain the map $K^T_*(\gr) \otimes_{R(T)} K^T_*(\flags) \cong
K^T_*(\gr \times \flags) \to K^T_*(\flags)$.  
We have the commutative diagram
\begin{align*}
\begin{diagram}
\node{K^T_*(\gr) \otimes_{R(T)} K^T_*(\flags)}\arrow {s,t}{\tocentralizer \otimes \tonilHecke} \arrow {e} \node{K^T_*(\flags)}  \arrow {s,b}{{\tonilHecke}}
\\ \node{\LL \otimes_{R(T)} \bK} \arrow {e,b}{\text {mult}} \node{\bK} 
\end{diagram}
\end{align*}

%\fixit{I removed the $\Delta'$ thing, which cluttered the notation.  %If it is confusing, we can add it back.  I'm also tempted to get rid %of $p^*$ in a lot of places.}
Via the pairing \eqref{eq:integral pairing} the dual map is the coproduct 
$$\Delta: K_T^*(\flags)\longrightarrow K_T^*(\gr) \otimes_{R(T)} K_T^*(\flags).
$$
%$$\Delta: K_T^*(\flags)\overset{\Delta'}{\longrightarrow} K_T^*(\gr) \otimes_{R(T)} K_T^*(\flags) \overset{p^* \otimes 1}{\longrightarrow} K_T^*(\flags) \otimes_{R(T)} K_T^*(\flags).$$
Note that $\Delta|_{K_T^*(\gr)}$ is the usual coproduct of $K_T^*(\gr)$, part of the $R(T)$-Hopf algebra structure of $K_T^*(\gr)$, and abusing notation we often denote $\Delta|_{K_T^*(\gr)}$ by $\Delta$.  Often, we will think of the image of $\Delta$ inside $K_T^*(\flags) \otimes_{R(T)} K_T^*(\flags)$ via the inclusion $p^*:K_T^*(\gr) \to K_T^*(\flags)$.

\begin{prop}\label{P:coprod} For all $a\in\LL$, $b\in\bK$, and $\psi \in K_T^*(\flags)$ we have
	\begin{align}\label{E:coprod}
		\pair{a b}{\psi} &= \sum_{(\psi)} \pair{a}{\psi_{(1)}} \pair{b}{\psi_{(2)}} 
	\intertext{where}
	\label{eq:sweedler}
	\Delta(\psi) &= \sum_{(\psi)} \psi_{(1)} \otimes \psi_{(2)}.
	\end{align}
\end{prop}
\begin{proof}
By definition and using \eqref{eq:sweedler} we have $\pair{a b}{\psi} = \sum_{(\psi)} \pair{a}{\psi^\gr_{(1)}}_{\gr} \pair{b}{\psi_{(2)}} $ where $\pair{a}{\psi}_{\gr}$ is the pairing between $\LL$ and $K^*_T(\gr)$ induced by Theorem \ref{T:centralizer} and the duality between $K_*^T(\gr)$ and $K^*_T(\gr)$.  But then since $\varphi^* \circ p^*$ is the identity, we have that $\pair{a}{\psi^\gr}_\gr =
\pair{a}{(\varphi^* \circ p^*)(\psi^\gr)}_{\gr} =
%\pair{\varphi_*(a)}{p^*(\psi^\gr)} =
\pair{a }{p^*(\psi^\gr)}$. In the second equality, we have used the projection formula
\begin{align}\label{E:projectionformula}
	\pair{a}{\varphi^*(b)}_{\thickaffgr} = \pair{\varphi_*(a)}{b}_{\flags}
	\qquad\text {for $a\in K_*^T(\gr)$, $b\in K_T^*(\flags)$.}
\end{align}
This gives the desired formula.
\end{proof}

\begin{lem}\label{L:coprod}
Let $\psi \in K_T^*(\flags)$.  If $\Delta(\psi) = \sum_{(\psi)} \psi_{(1)} \otimes \psi_{(2)}$, then $\theta(\psi_{(1)}) = \psi_{(1)}$.
\end{lem}
\begin{proof} This follows from the fact that the elements in the first tensor factor are in fact in the image of $K_T^*(\gr)$ inside $K_T^*(\flags)$.
\end{proof}

\subsection{Loop evaluation at identity}
Let $\ev_1:LK/\TR\to K/\TR$ be induced by evaluation of a loop at the identity. Since this is a $\TR$-equivariant map (via left translation) it induces an $R(T)$-algebra homomorphism 
\begin{equation}\label{eq:ev1}
\ev_1^*:K_T^*(G/B)\to K_T^*(\flags).
\end{equation}
%Note that $LK = (\Omega K) K$ where $K$ is regarded as the subspace of $LK$ of constant loops. 
Let $q: K/\TR \to LK/\TR$ be the natural inclusion; it is $\TR$-equivariant for left translation.  The algebraic analogue of $q$ identifies $G/B$ with the finite-dimensional Schubert variety $X_{w_0} \subset \flags$.

Define $\eta := \ev_1^* \circ q^*$ so that $\eta: K^*_T(\flags) \to K^*_T(\flags)$ is the pullback map in equivariant $K$-theory of the following composition
\begin{align}\label{eq:eta}
LK/\TR \xrightarrow{\;\;\;\ev_1\;\;\;} K/\TR \xrightarrow{\;\;\;q\;\;\;} LK/\TR.
\end{align}

\begin{lem} \label{L:evoneloc} For all $\mu\in Q^\vee$, $v\in\Wf$,
and $\psi\in K_T^*(\flags)$ we have
\begin{align}\label{E:evoneloc}
	(\eta\psi)(t_\mu v) = \psi(v).
\end{align}
\end{lem}
\begin{proof} Recalling the description of the based loop defined by $t_\mu$ from the proof of Lemma \ref{L:wrongwayloc}, evaluating the loop $t_\mu v$ at the identity yields the value $v$.
Thus the $\TR$-fixed point $t_\mu v$ is sent to $v$ under the composition \eqref{eq:eta}.
\end{proof}

\begin{lem} \label{L:evonelinebundle} For all $\la\in P$,
	\begin{align}\label{E:evonelinebundle}
		\ev_1^*([\cL_\la^{G/B}]) = [\cL_\la].
	\end{align}
\end{lem}
\begin{proof} For all $\mu\in Q^\vee$ and $u\in W$ we have
	\begin{align*}
		i_{t_\mu v}^*(\ev_1^*([\cL_\la^{G/B}])) 
		&= i_v^*([\cL_\la^{G/B}]) 
		= v \cdot e^\la 
		= (t_\mu v) \cdot e^\la 
		= i_{t_\mu v}^*([\cL_\la]). \qedhere
	\end{align*}
\end{proof}

\subsection{Coproduct identity}
The following identity is the main result of this section.
\begin{prop}\label{P:coprodbulletK}
For $\psi \in K^*_T(\flags)$ and $a \in \bK$, we have
$$
a \bullet \psi = \sum_{(\psi)} \psi_{(1)} \cup \eta(a \bullet \psi_{(2)})
$$
where $\Delta(\psi) = \sum_{(\psi)} \psi_{(1)} \otimes \psi_{(2)}$.
In particular, taking $a = 1$, we have the identity
$$
\cup \circ (1 \otimes \eta) \circ \Delta = 1
$$
in $\End_{R(T)}(K^*_T(\flags))$.
\end{prop}
\begin{proof}
For $\mu \in Q^\vee$ and $v \in \Wf$, we compute
\begin{align*}
\pair{t_\mu v}{a \bullet \psi}  
&= \pair{t_\mu v a}{\psi} \\
&= \sum_{(\psi)} \pair{t_\mu}{\psi_{(1)}} \pair{va}{\psi_{(2)}}& \mbox{by Proposition \ref{P:coprod}} \\
& = \sum_{(\psi)} \pair{t_\mu}{\psi_{(1)}} \pair{v}{a \bullet \psi_{(2)}} \\
& = \sum_{(\psi)} \pair{t_\mu v}{\psi_{(1)}} \pair{t_\mu v}{\eta(a \bullet \psi_{(2)})} & \mbox{by Lemmas \ref{L:wrongwayloc}, \ref{L:coprod} and \ref{L:evoneloc}} \\
&= \pair{t_\mu v}{ \sum_{(\psi)} \psi_{(1)} \cup \eta(a \bullet \psi_{(2)})}. &&\qedhere
\end{align*}
\end{proof}

\subsection{Commutation relations}
We record additional commutation relations involving the nilHecke algebra actions, and the endomorphisms $\theta$ and $\eta$.

Let $\kappa:K^*_T(\flags) \to K^*_T(\flags)$ be the pullback map in equivariant $K$-theory induced by the composition
\begin{align} \label{eq:kappa}
\flags \longrightarrow \id \longrightarrow \flags
\end{align}
where $\id$ denotes the basepoint of $\flags$. It is an $R(T)$-algebra homomorphism.

\begin{lem} \label{L:kappaloc} For all $\mu\in Q^\vee$, $v\in\Wf$,
and $\psi\in K_T^*(\gr)$ we have
\begin{align}\label{E:kappaloc}
	\kappa(\psi)(t_\mu v) = \psi(\id).
\end{align}
\end{lem}

\begin{lem}\label{L:endo}
As $R(T)$-module endomorphisms of $K^*_T(\flags)$, we have the relations
$$
\theta^2 = \theta, \qquad \eta^2 = \eta, \qquad \kappa^2 = \kappa;
$$
$$
\theta \eta =  \eta \theta = \theta \kappa = \kappa \theta = \eta \kappa = \kappa \eta = \kappa.
$$

\end{lem}
\begin{proof}
Straightforward from Lemmas \ref{L:wrongwayloc}, \ref{L:evoneloc}, and \ref{L:kappaloc}.
\end{proof}

For $w \in \Wa$, define the endomorphism $$w \odot := (w \cdot) \circ (w \bullet) = (w \bullet) \circ (w \cdot)$$ of $K^*_T(\flags)$. 

\begin{prop}\label{P:theta}
The map $\theta$ interacts with the two actions $\cdot$ and $\bullet$ of $\bK$ on $K^*_T(\flags)$ in the following way:
\begin{enumerate}
\item
$(q \cdot) \circ \theta = \theta \circ (q\cdot)$
\item
$(t_\mu \cdot) \circ \theta = \theta \circ (t_\mu \cdot)$
\item
$(w \cdot) \circ \theta = \theta \circ (w \odot)$
%\item
%$(q \bullet) \circ \theta = $ no good formula
%\item
%$(t_\mu \bullet) \circ \theta$ can be computed by Theorem \ref{T:coprod}
\item
$(w \bullet) \circ \theta = \theta$
\end{enumerate}
where $q \in R(T)$, $w \in \Wf$, and $\mu\in Q^\vee$.  By (1), (2), (3), we see that $\theta(K^*_T(\flags)) = p^*(K^*_T(\gr))$ is a $\bK$-submodule of $K^*_T(\flags)$ under the $\cdot$ action.
\end{prop}

\begin{prop}\label{P:eta}
The map $\eta$ interacts with the two actions $\cdot$ and $\bullet$ of $\bK$ on $K^*_T(\flags)$ in the following way:
\begin{enumerate}
\item
$(q \cdot) \circ \eta = \eta \circ (q\cdot) $
\item
$(t_\mu \cdot) \circ \eta = \eta $
\item
$(w \cdot) \circ \eta = \eta \circ (w \cdot)$
\item
$(q \bullet) \circ \eta = \eta \circ (q\bullet)$
\item
$(t_\mu \bullet) \circ \eta = \eta$
\item
$(w \bullet) \circ \eta = \eta \circ (w\bullet)$ 
\end{enumerate}
where $q \in R(T)$, $w \in \Wf$, and $\mu\in Q^\vee$.  By (1)-(6) we see that $\eta(K^*_T(\flags)) = \ev_1^*(K^*_T(G/B))$ is a $\bK$-submodule of $K^*_T(\flags)$ under either the $\cdot$ or the $\bullet$ action.
\end{prop}

\begin{prop}\label{P:kappa}
The map $\kappa$ interacts with the two actions $\cdot$ and $\bullet$ of $\bK$ on $K^*_T(\flags)$ in the following way:
\begin{enumerate}
\item
$(q \cdot) \circ \kappa = \kappa \circ (q\cdot)$
\item
$(t_\mu \cdot) \circ \kappa = \kappa$
\item
$(w \cdot) \circ \kappa = \kappa \circ (w \odot)$
\item
$(t_\mu \bullet) \circ \kappa = \kappa$
\item
$(w \bullet) \circ \kappa = \kappa$
\end{enumerate}
where $q \in R(T)$, $w \in \Wf$, and $\mu\in Q^\vee$.
\end{prop}

\subsection{Action of $\bK$ on tensor products}
Define $\bK_{Q(T)} \otimes_{Q(T)} \bK_{Q(T)}$ to be the left $Q(T)$-bilinear tensor product such that
\begin{align}
	q (a \otimes b) = qa \otimes b = a \otimes q b
\end{align}
for all $a,b\in\bK_{Q(T)}$ and $q\in Q(T)$. Define $\Delta: \bK_{Q(T)}\to \bK_{Q(T)} \otimes_{Q(T)} \bK_{Q(T)}$ by
\begin{align}\label{E:Delta}
	\Delta(\sum_{w\in\Wa} a_w w) = \sum_w a_w w\otimes w
\end{align}
for $a_w\in Q(T)$.
Then for all $i\in\Iaf$ we have
\begin{align}\label{E:deltaA}
	\Delta(T_i) &= T_i \otimes 1 + 1 \otimes T_i + (1-e^{\al_i})T_i \otimes T_i.
\end{align}
This restricts to a left $R(T)$-bilinear tensor product $\Delta: \bK \to \bK\otimes_{R(T)} \bK$.
If $M$ and $N$ are left $\bK$-modules then $M \otimes_{R(T)} N$ is a left $\bK$-module via 
\begin{align}\label{E:tensordef}
	a(m \otimes n) &= \sum_{(a)} a_{(1)}(m) \otimes a_{(2)}(n)
\end{align}
for all $a\in\bK$, $m\in M$ and $n\in N$.

\begin{lem}\label{L:cupproduct}
For $\psi_1,\psi_2 \in K^*_T(\flags)$ and $a \in \bK$, we have
\begin{align*}
a\cdot(\psi_1 \cup \psi_2) &= \sum_{(a)} (a_{(1)} \cdot \psi_1) \cup (a_{(2)} \cdot \psi_{(2)}) \\
a\bullet(\psi_1 \cup \psi_2) &= \sum_{(a)} (a_{(1)} \bullet \psi_1) \cup (a_{(2)} \bullet \psi_{(2)}).
\end{align*}
\end{lem}
\begin{proof}
We have
\begin{align*}
%w\cdot(\psi_1 \cup \psi_2)(x) &= \up{w}(\psi_1(w^{-1} x) \psi_2(w^{-1}(x)) %=  %\up{w}(\psi_1(w^{-1} x)) \up{w}(\psi_2(w^{-1} x)) 
w\cdot(\psi_1 \cup \psi_2)(x) &= w(\psi_1(w^{-1} x) \psi_2(w^{-1}(x)))
= ((w \cdot \psi_1) \cup (w \cdot \psi_2))(x) \\
w\bullet(\psi_1 \cup \psi_2)(x) &= \psi_1(xw) \psi_2(xw) =  ((w \bullet \psi_1) \cup (w \bullet \psi_2))(x),
\end{align*}
consistent with $\Delta(w) = w \otimes w$.  Next, we check that the formulae are compatible with $R(T)$-linearity. It is enough to work with the algebra generators $e^\la$ of $R(T)$. We have $\Delta(e^\la w) = e^\la w \otimes w$ and
\begin{align*}
	(e^\la w) \cdot (\psi_1\cup \psi_2) &=e^ \la \cdot ( w \cdot (\psi_1\cup\psi_2))
	= e^\la \cdot ((w\cdot \psi_1)\cup (w\cdot \psi_2)) 
	= ((e^\la w) \cdot \psi_1)\cup (w\cdot \psi_2).
\end{align*}
Using Lemma \ref{L:Xbullet} we have
\begin{align*}
	(e^\la w) \bullet (\psi_1\cup \psi_2) &= 
	e^\la \bullet (w \bullet (\psi_1\cup\psi_2)) 
	=[\cL_\la] \cup (w \bullet \psi_1) \cup (w\bullet \psi_2) 
	= ((e^\la w)\bullet \psi_1 )\cup (w\bullet \psi_2).
\end{align*}
\end{proof}
%\Seungjin{ Isn't $(e^\la w) \bullet f = w \bullet (e^\la \bullet f)$?}
%\MarkS{No. $\bullet$ is actually a left action and
%	$e^\la$ and $w$ do not commute.}

% If $B$ is a commutative $R(T)$-algebra then we have
%\begin{align}
%	a(b_1 b_2) &= \sum_{(a)} a_{(1)}(b_1) a_{(2)}(b_2)
%\end{align}
%which is obtained from the action of $\bK$ on $B \otimes_{R(T)} B$ followed by the multiplication map $ B \otimes_{R(T)} B\to B$.

\subsection{Finite nilHecke algebra}
The finite nilHecke ring $\bKfin$ is the subring of $\bK$ generated by $R(T)$ and $T_i$ for $i\in \Ifin$. 
There are left actions $\cdot$ and $\bullet$ of $\bKfin$ on $K_T^*(G/B)$ that are similarly to the actions of $\bK$ on $K_T^*(\flags)$.

There is a $\bKfin$-$\bKfin$-bimodule and ring homomorphism $\cl_{\bK}: \bK \to \bKfin$ defined (for convenience from $\bK_{Q(T)}\to \bKfin_{Q(T)}$) by
\begin{align}\label{E:clbA}
	\cl_{\bK}(t_\mu a) = a \qquad\text {for $\mu\in Q^\vee$ and $a\in\bKfin$.}
\end{align}
In particular,
\begin{align*}
	\cl_{\bK}(T_0) &= \cl_{\bK}((1-e^{-\theta})^{-1}(s_0 - 1)) 
	= \cl_{\bK}((1-e^{-\theta})^{-1}(t_{\theta^\vee} s_\theta-1)) 
	=(1-e^{-\theta})^{-1}(s_\theta - 1) =: T_{-\theta}.
\end{align*}
Thus we have $\cdot$ and $\bullet$ actions of $\bK$ on $K^*_T(G/B)$ that factor through $\cl_{\bK}: \bK \to \bKfin$.

\subsection{Tensor product decomposition of $K^*_T(\flags)$}
The equivariant $K$-theory ring $K_T^*(\gr)$ is a left $\bK$-submodule of $K_T^*(\flags)$ under the $\cdot$-action.  Thinking of $\psi_\gr \in K_T^*(\gr)$ as a function from cosets $\Wa/\Wf$ to $R(T)$, we have $(w \cdot \psi_\gr)(x\Wf) = %\up{w}(\psi_\gr(w^{-1}x  \Wf))$.
w(\psi_\gr(w^{-1}x  \Wf))$.

The left $\bK$-module structures via $\cdot$ on $K^*_T(\gr)$ and $K^*_T(G/B)$ give a left $\bK$-module structure on $K_T^*(\gr) \otimes_{R(T)} K_T^*(G/B)$ via \eqref{E:tensordef}.
\begin{thm}\label{T:tensor}
There is an $R(T)$-algebra isomorphism
\begin{align}
\label{E:tensormap}
  K_T^*(\gr) \otimes_{R(T)} K_T^*(G/B) &\cong K_T^*(\flags) \\
\label{E:tensormaponelements}
  a \otimes b& \mapsto p^*(a) \cup \ev_1^*(b)
\end{align}
with componentwise multiplication on the tensor product. This map is also an isomorphism of left $\bK$-modules under the $\cdot$ action.
\end{thm}
The proof is delayed to after Theorem \ref{T:affineschub}.  

\section{Affine Schubert classes}\label{sec:Schubert}
\subsection{Schubert bases}\label{ssec:Schubert}
The $R(T)$-algebras $K^*_T(\flags)$, $K^*_T(\gr)$, and $K^*_T(G/B)$ have equivariant Schubert bases $\{\psi^x \mid x \in \Wa\}$, $\{\psi^u_{\gr} \mid u \in \Wz\}$, and $\{\psi^w_{G/B} \mid w \in \Wf\}$ respectively.   The basis $\{\psi^x \mid x \in \Wa\} \subset K^*_T(\flags)$ is uniquely characterized by
\begin{equation} \label{E:psidef}
 \psi^v(T_w) = \delta_{v,w}.
\end{equation}

We have
\begin{align}
\label{E:pschub}
	p^*(\psi^z_{\thickaffgr}) &= \psi^z\qquad\text {for all $z\in\Wz$,}\\
\label{E:qschub}
	q^*(\psi^x) &= \begin{cases}\psi^x_{G/B} & \text {for $x\in\Wf$,} \\
0 &  \text {for $x\in\Wa \setminus \Wf$.}
\end{cases}
\end{align}
In particular $\eta(\psi^x)=\ev_1^*(\psi^x_{G/B})$ for $x\in\Wf$.

Similarly, let $\{\psi_x \mid x \in \Wa\}$, $\{\psi_u^{\gr} \mid u \in \Wz\}$, and $\{\psi_w^{G/B} \mid w \in \Wf\}$ denote homology Schubert bases of $K_*^T(\flags)$, $K_*^T(\gr)$, and $K_*^T(G/B)$.  We write $\pair{\cdot}{\cdot}_{\flags}$, $\pair{\cdot}{\cdot}_{\gr}$, and $\pair{\cdot}{\cdot}_{G/B}$ for the $R(T)$-bilinear pairings between $T$-equivariant $K$-homology and $K$-cohomology, so that for example $\pair{\psi_x}{\psi^y}_{\flags} = \delta_{xy}$.  For the precise geometric interpretations of $\psi^x$ and $\psi_x$ we refer the reader to \cite[\S3]{LLMS}.

\begin{rem}  \label{R:pproj}
The map $p^*$ is an isomorphism of $K_T^*(\thickaffgr)$ with its image
$\bigoplus_{u\in\Wz} R(T) \psi^u$, whose elements are $\Wf \bullet$-invariant by Proposition \ref{P:theta}.
\end{rem}

The localization values of Schubert classes are determined by the following triangular relation. For all $w\in\Wa$, in $\bK$ we have \cite{KK} \cite[Proposition 2.4]{LSS}
\begin{align}\label{E:Schubloc}
	w = \sum_{v\le w} \pair{w}{\psi^v} T_v.
\end{align}

The Schubert basis $\{\psi^w \mid w \in \Wa\}$ interacts with the $\cdot$ and $\bullet$ actions of $\bK$ as follows.   For $i \in \Iaf$, define 
\begin{align}
y_i&:= 1 + T_i = \frac{1}{1-e^{-\alpha_i}}(1-e^{-\alpha_i}s_i) & \ty_i &:= 1-e^{\alpha_i}T_i = \frac{1}{1-e^{\alpha_i}}(1-e^{\alpha_i} s_i).
\end{align}

\begin{prop}\label{P:dotbullet} 
For $\la\in \Xfin$ and $T_i$ for $i\in\Iaf$, on the Schubert basis element $\psi^w \in K_T^*(\flags)$ for $w\in\Wa$, we have:
\begin{align}
	\label{E:Aleft}
	\ty_i \cdot \psi^w &= 
	\begin{cases}
		\psi^{s_iw}& \text {if $s_iw<w$} \\
		\psi^{w} & \text {otherwise.}
	\end{cases} \\
	\label{E:Aright}
	y_i \bullet \psi^w &= 
	\begin{cases}
		\psi^{ws_i} & \text {if $ws_i<w$} \\
		\psi^w & \text {otherwise.}
	\end{cases} \\
	\label{E:Sleft}
	e^\la \cdot \psi^w &= e^\la \psi^w \\
	\label{E:Sright}
	e^\la \bullet \psi^w &= [\cL_\la] \cup \psi^w  .
\end{align}
\end{prop}

\begin{proof}
\eqref{E:Aright} is \cite[Lemma 2.2]{LSS}. Equation \eqref{E:Aleft} has a straightforward proof starting with $\pair{T_v}{\tilde{y}_i \cdot \psi^w}$ and using \eqref{eq:Tonfunc} and the duality of the two bases $\{T_v\}$ with $\{ \psi^w\}$. 

Equation \eqref{E:Sleft} follows from the definition and \eqref{E:Sright} follows from \eqref{E:Xbullet}.
\end{proof}

\subsection{Equivariant affine $K$-Stanley classes} 
Theorem \ref{T:centralizer} interacts with Schubert classes as follows.

\begin{thm}\label{T:centralizerSchub} \cite[Theorem 5.4]{LSS}
For every $u\in \Wz$, $k_u := \tocentralizer(\psi_u^{\gr})$ is the unique element of $\LL$ of the form 
\begin{align}\label{E:jbasis}
%k_w = T_w + \sum_{z\in \Wa\setminus\Wz} k^z_w T_z
k_u = \sum_{z\in \Wa} k^z_u T_z
\end{align}
for some $k^z_u\in R(T)$, where 
\begin{align}
k^z_u=\delta_{z,u}\qquad\text{for $z\in\Wz$.}
\end{align}
\end{thm}

\begin{rem}\label{R:clAj} 
	%This can be proved using Theorem \ref{T:centralizer} and \eqref{E:Schubloc}.
	It follows from Theorem \ref{T:centralizer} that
	\begin{align}
		t_\mu = \sum_{u\in \Wz} \pair{t_\mu}{\psi^u} k_u.
	\end{align}
Taking $\cl_{\bK}$ of both sides and using \eqref{E:clbA}, we have
\begin{align}
	T_\id &= \sum_{u\in \Wz} \pair{t_\mu}{\psi^u} \cl_{\bK}(k_u).
\end{align}
Now $\pair{t_\mu}{\psi^u}$ is zero unless $ u\le t_\mu$,
which by the assumption $u \in \Wz$ is equivalent to
$u\Wf \le t_\mu \Wf$. Since both the $t_\mu$ and the $k_u$
are $Q(T)$-bases of $\mathbb{L}$ it follows that $\pair{t_\mu}{\psi^u}\ne0$
for $\mu\in Q^\vee$ and $u\in \Wz$ such that $t_\mu \Wf=u\Wf$. It follows by induction that
\begin{align}
	\label{E:clAj}
	\cl_{\bK}(k_x) &= \delta_{\id,x} T_{\id}\qquad\text{for all $x\in \Wz$.}
\end{align}
\end{rem}

For $w\in\Wa$ the
equivariant affine $K$-Stanley class $G^w \in K_T^*(\gr)$ is defined by
\begin{align}\label{E:affstan}
	  G^w := \varphi^*(\psi^w).
\end{align}
We will also consider $G^w$ an element of $K_T^*(\flags)$ via $p^*$. 
%Thus $G^w) = \theta(\psi^w)$. 
% In particular, for $u \in \Wz$, we have
%\begin{align} \label{E:GWz}
%p^*(G^u) &= \psi^u.
%\end{align}

\begin{lem}\label{L:Fcoeffs} For $w\in \Wa$, we have
\begin{align}\label{E:Fcoeffs}
	G^w &= \sum_{u\in \Wz} k^w_u \,\psi^u_{\gr}
\end{align}
where the $k^w_u$ are defined in Theorem \ref{T:centralizerSchub}.
\end{lem}
\begin{proof} 
For $u\in \Wz$, by \eqref{E:projectionformula} and Theorems \ref{T:centralizer} and \ref{T:centralizerSchub} we have
\begin{align*}
	\pair{\psi_u^{\gr}}{G^w}_{\thickaffgr} &= \pair{\psi_u^{\gr}}{\varphi^*(\psi^w)}_{\thickaffgr} 
	= \pair{\varphi_*(\psi_u^{\gr})}{\psi^w}_{\flags} 
	= \pair{\sum_{z\in \Wa} k^z_u T_z}{\psi^w} 
	=  k^w_u. \qedhere
\end{align*}
\end{proof}
Recall that $u*v$ denotes the Demazure product of $u$ and $v$.

\begin{prop}\label{P:coprod1} For $w\in \Wa$, we have
	\begin{align}\label{E:coprod1}
		\Delta(\psi^w) &= \sum_{w=w_1 * w_2 } (-1)^{\ell(w_1)+\ell(w_2)-\ell(w)} G^{w_1} \otimes \psi^{w_2}.
	\end{align}
\end{prop}
\begin{proof}
For $u\in\Wz$ and $v\in\Wa$, we have
\begin{align}\label{E:kuTv}
k_u T_v = \sum_{x\in\Wa} k^x_u T_x T_v
	= \sum_{w \in \Wa} \sum_{\substack{x \in \Wa \\ w = x*v}}
	(-1)^{\ell(x)+\ell(v) - \ell(w)} k^{x}_u T_w.
\end{align}
This gives a formula for the matrix of the multiplication map $\LL \otimes_{R(T)} \bK \to \bK$ with respect to the bases
$k_u \otimes T_v$ and $T_w$. 
The dual map $K_T^*(\flags)\overset{\Delta}{\longrightarrow} K_T^*(\thickaffgr) \otimes_{R(T)} K_T^*(\flags)$ has the transposed matrix of Schubert matrix coefficients.
That is, for all $w\in\Wa$, using Lemma \ref{L:Fcoeffs} we have
\begin{align*}
\Delta(\psi^w) &= 
\sum_{(u,v)\in\Wz\times\Wa} \sum_{\substack{x\in\Wa \\ w=x*v}} (-1)^{\ell(x)+\ell(v)-\ell(w)} k^x_u \psi_\gr^u \otimes \psi^v \\
&=\sum_{\substack{v,x\in\Wa \\ w=x*v}} (-1)^{\ell(x)+\ell(v)-\ell(w)} G^x \otimes \psi^v.
\end{align*}
\end{proof}

\subsection{Coproduct formula for affine Schubert classes}
The following formula decomposes $\psi^w$ according to the tensor product isomorphism of Theorem \ref{T:tensor}.

\begin{thm} \label{T:affineschub} For $w\in\Wa$, we have
	\begin{align}\label{E:affineschub}
		\psi^w &= \sum_{\substack{(w_1,w_2)\in \Wa\times\Wf \\ w_1*w_2=w}} (-1)^{\ell(w_1)+\ell(w_2) - \ell(w)} G^{w_1} \cup \ev_1^*(\psi^{w_2}_{G/B}) 
		\end{align}
%	\begin{align}\label{E:affineschub}
%		\psi^w &= \sum_{\substack{(w_1,w_2)\in \Wa\times\Wf \\ w_1*w_2=w}} (-1)^{\ell(w_1)+\ell(w_2) - \ell(w)} p^*(G^{w_1}) \cup \ev_1^*(\psi^{w_2}_{G/B}) \\
%&=\sum_{\substack{(w_1,w_2)\in \Wa\times\Wf \\ w_1*w_2=w}} (-1)^{\ell(w_1)+\ell(w_2) - \ell(w)}\theta(\psi^{w_1}) \cup \eta(\psi^{w_2}).
%	\end{align}
\end{thm}

\begin{proof}%Proof of Theorem \ref{T:affineschub}}
Apply Proposition \ref{P:coprodbulletK} with $a = 1$ and $\psi = \psi^w$, and use Proposition \ref{P:coprod1}.
\end{proof}

\noindent \begin{proof}[Proof of Theorem \ref{T:tensor}]
As $p^*$ and $\ev_1^*$ are $R(T)$-algebra homomorphisms, so is \eqref{E:tensormaponelements}. Note that for $u\in\Wz$, $G^u=\psi^u_{\thickaffgr}$.  To show that \eqref{E:tensormaponelements} is an isomorphism, it suffices to show that the image of the basis $\{ \psi^u_{\thickaffgr} \otimes \psi^v_{G/B} \mid (u,v)\in \Wz\times \Wf\}$ of $K_T^*(\thickaffgr)\otimes_{R(T)} K_T^*(G/B)$, namely, $\{G^u \cup \ev_1^*(\psi^v_{G/B})\mid (u,v)\in\Wz\times\Wf\}$, is an $R(T)$-basis of $K_T^*(\flags)$. But the latter collection of elements is unitriangular with the Schubert basis of $K_T^*(\flags)$, by Theorem \ref{T:affineschub}.  Thus \eqref{E:tensormaponelements} is a $R(T)$-algebra isomorphism.

Finally, \eqref{E:tensormap} is a left $(\bK\,\cdot)$-module homomorphism, due to Lemma \ref{L:cupproduct} and the fact that $\ev_1^*$ and $p^*$ are left $(\bK\,\cdot)$-module homomorphisms.
\end{proof}
%\subsection{Formulae for Schubert divisors}

\begin{cor} \label{C:affschubspecial} 
For $i\in\Iaf$, we have
\begin{align}\label{E:affschubreflection}
	\psi^{s_i} &= 
	\begin{cases}	
		G^{s_0} &\text {if $i=0$} \\
		G^{s_i} + \ev_1^*(\psi^{s_i}_{G/B}) - G^{s_i}\cup \ev_1^*(\psi^{s_i}_{G/B}) & \text {otherwise.}
	\end{cases}
\end{align}
\end{cor}

\begin{prop}\label{P:reflectionclass} For all $i\in\Iaf$ we have
	\begin{align}
		\label{E:Greflection}
		1-G^{s_i} &= (1- G^{s_0})^\ell
	\end{align}
where $\ell=\level(\La_i)$.
\end{prop}

\begin{proof}
Let $\{\psi^{x}_{\Taf} \mid x \in \Wa\}$ denote the equivariant Schubert basis of $K^*_{\Taf}(\flags)$, where $\Taf \cong T \times \C^\times$ denotes the affine maximal torus.
For all $i\in\Iaf$, in $K^*_{\Taf}(\flags)$ we have \cite{KS}\footnote{The conventions here differ by a sign to those in \cite{KS}. For example, for us $\psi^{s_i}(s_i)=1-e^{\alpha_i}$.}
\begin{align*}
\psi^{s_i}_{\Taf}(w)&= 1-e^{\La_i - w \cdot \La_i}\qquad\text {for all $w\in\Wa$.}
\end{align*}
%or equivalently
%\begin{align}\label{E:divisorlinebundleaffine}
%  \xi^{s_i}_{\Taf} =  c_1(\cL^{\flags,\Taf}_{\La_i})-\La_i.
%\end{align}
For all $\mu\in Q^\vee$ and $v\in\Wf$ we have
\begin{align*}
	G^{s_i}(t_\mu v)
 =\psi^{s_i}(t_\mu) 
 = \cl(\psi^{s_i}_{\Taf}(t_\mu)) 
 = \cl(1-e^{\La_i-t_{\mu} \cdot \La_i}).
\end{align*}
Applying this equation twice, we have
\begin{align*}
 {1-G^{s_i}(t_\mu v)  \over (1- G^{s_0}(t_\mu v))^\ell}   = \cl(e^{\La_i-t_{\mu} \cdot \La_i-\ell\La_0+\ell t_{\mu} \cdot \La_0})
  = \cl(e^{\af(\omega_i)-t_\mu \cdot \af(\omega_i)}) 
  = 1
\end{align*}
since for any level zero element $\la$ we have $t_\mu(\la) = \la - \pair{\mu}{\la} \delta$.
\end{proof}

\subsection{Ideal sheaf classes}
\def\bpsi{\bar \psi}
\def\bG{\bar G}
For a reduced word $w = s_{i_1} \cdots s_{i_\ell}$, define 
$
y_w:= y_{i_1} \cdots y_{i_\ell} \in \bK,
$
which does not depend on the choice of reduced word.
By \cite[Lemma A.3]{LSS}, we have $y_w = \sum_{v \leq w} T_v$.
We let $\{\bpsi^w \in K^*_T(\flags) \mid w \in \Wa\}$ denote the dual basis to $\{y_w \mid w \in \Wa\}$.  Thus $\ip{y_w}{\bpsi^v} = \delta_{w,v}$.  The element $\bpsi^w$ is denoted $\psi_{KK}$ in \cite{LSS}.

\begin{remark}
The Schubert basis element $\psi^w \in K^*_T(\flags)$ represents the class of the structure sheaf $\cO_w$ of a Schubert variety in the thick affine flag variety.  The element $\bpsi^w$ represents the ideal sheaf ${\mathcal I}_w$ of the boundary $\partial X_w$ in a Schubert variety $X_w$.  See \cite[Appendix A]{LSS}.
\end{remark}

Define 
$$
\bG^w:= \varphi^*(\bpsi^w) \in K^*_T(\gr)
$$
and as usual, we denote by $\bG^w$ the image of this element in $K^*_T(\flags)$.
Following \cite{LLMS}, define $l_u:= \sum_{v \in \Wz: v \leq u} k_v \in \LL$ and define $l_u^w \in R(T)$ by
\begin{equation}\label{E:lbasis}
l_u := \sum_{w \in \Wa} l_u^w y_w.
\end{equation}
The coefficients $l_u^w$ are related to $k_v^x$ by the formula
$$
l_u^w = \sum_{x \leq w} (-1)^{\ell(w)-\ell(x)} \sum_{\substack{v \in \Wz \\ v \leq u}} k_v^x.
$$
We have the following variants of Lemma \ref{L:Fcoeffs}, Proposition \ref{P:coprod1}, and Theorem \ref{T:affineschub} with identical proofs.

\begin{lem}\label{L:bGcoeffs} For $w\in \Wa$, we have
\begin{align}\label{E:bGcoeffs}
	\bG^w &= \sum_{u\in \Wz} l^w_u \,\bpsi^u_{\gr}
\end{align}
where the $l^w_u$ are defined in \eqref{E:lbasis}, and $\bpsi^u_{\gr}$ is determined by $\ip{l_v}{\bpsi^u_{\gr}} = \delta_{v,u}$. 
\end{lem}

\begin{prop} For $w\in \Wa$, we have
	\begin{align*}
		\Delta(\bpsi^w) &= \sum_{w=w_1 * w_2 } \bG^{w_1} \otimes \bpsi^{w_2}.
	\end{align*}
\end{prop}

\begin{thm} For $w\in\Wa$, we have
	\begin{align*}
		\bpsi^w &= \sum_{\substack{(w_1,w_2)\in \Wa\times\Wf \\ w_1*w_2=w}} \bG^{w_1} \cup \ev_1^*(\bpsi^{w_2}_{G/B}) 
%=\sum_{\substack{(w_1,w_2)\in \Wa\times\Wf \\ w_1*w_2=w}} \theta(\bpsi^{w_1}) \cup \eta(\bpsi^{w_2}).
	\end{align*}
\end{thm}

\subsection{$\bullet$-action on affine Schubert classes}
We investigate the behavior of the decomposition in Theorem \ref{T:affineschub} under the $\bullet$-action of $\bK$.  By Lemma \ref{L:cupproduct} and Theorem \ref{T:tensor}, it is enough to separately describe how $\theta(\psi^x)$ for $x\in\Wa$, and $\eta(\psi^w)$ for $w\in\Wf$, behave under the $\bullet$-action.  
For $\eta(\psi^w)$, Proposition \ref{P:eta} gives the following.

\begin{prop} \label{P:nilHeckeeta} For $a\in\bK$ and $w\in \Wf$, we have
\begin{align}
%	a \cdot \eta(\psi^w) &= \eta(\cl_{\bK}(a)\cdot \psi^w)  \\
	a\bullet \eta(\psi^w) &= \eta(\cl_{\bK}(a)\bullet \psi^w)
\end{align}
and in particular,
\begin{align}
\label{E:s0bulleteta}
	s_0 \bullet \eta(\psi^w) &= \eta(s_\theta\bullet \psi^w) \\
\label{E:A0bulleteta}
	y_0 \bullet \eta(\psi^w) &= \eta(y_{-\theta} \bullet \psi^w)
\end{align}
where $y_{-\theta}:=1+ T_{-\theta}$.
\end{prop}

Since $\theta(\psi^x)$ can be expanded in the basis $\psi^u$ for $u\in\Wz$,
it is enough to consider the $\bullet$ action on $\psi^u$.

\begin{thm}\label{T:nilHecketheta}
For $\la\in \Xfin$, $i\in\Iaf$, and $u \in \Wz$, we have
\begin{enumerate}
\item
$e^\la \bullet \psi^u = [\cL_\la] \cup \psi^u$,
\item
$y_i \bullet \psi^u =\psi^u$ and $s_i\bullet \psi^u=\psi^u$ if $i \in I$,
\item For $u\in\Wz\setminus\{\id\}$, we have
\begin{align}\label{E:A0bulletSchub}
  y_0 \bullet \psi^u &= \psi^{us_0} 
  = \sum_{\substack{(x_1,x_2)\in\Wa\times\Wf \\ x_1*x_2=us_0}} (-1)^{\ell(x_1)+\ell(x_2) - \ell(u) -1}
	\theta(\psi^{x_1}) \cup \eta(\psi^{x_2}).
\end{align}
\item For $u\in\Wz\setminus \{\id\}$, we have
\begin{align}\label{E:s0bulletSchub}
	s_0\bullet \psi^u &= 
	e^{-\theta}\psi^u + \sum_{\substack{(x_1,x_2)\in\Wa\times\Wf \\ x_1*x_2=us_0}} (-1)^{\ell(x_1)+\ell(x_2) - \ell(u) - 1}
	\theta(\psi^{x_1}) \cup \eta((1-e^{-\theta})\bullet \psi^{x_2}).
\end{align}
\end{enumerate}
\end{thm}
\begin{proof} These formulae may be deduced from Proposition \ref{P:dotbullet} using $s_0 =e^{-\theta} + (1-e^{-\theta}) y_0$.
\end{proof}

\begin{rem}
The $\cdot$ and $\bullet$ actions of $\bK$ make $K_T^*(\thickaffgr) \otimes_{R(T)} K_T^*(G/B)$ into a left $(\bK\times\bK)$-module such that the map \eqref{E:tensormap} is a left $(\bK\times\bK)$-module isomorphism.
\end{rem}

\subsection{Recursion}
The affine Schubert classes in the tensor product $K^*_T(\gr) \otimes_{R(T)} K^*_T(G/B)$ are determined by the following recursion.
\begin{enumerate}
\item $\psi^u = \psi^u_\gr \otimes 1$ for $u \in \Wz$, and
\item For all $i\in\Iaf$, $$y_i \bullet \psi^w = 
	\begin{cases}
		\psi^{ws_i} & \text {if $ws_i<w$} \\
		\psi^w & \text {otherwise.}
	\end{cases}$$
\end{enumerate}
The operator $y_i$ acts on $K^*_T(\gr) \otimes_{R(T)} K^*_T(G/B)$ by
$$
\Delta(y_i) = (1-e^{\alpha_i})y_i\otimes y_i + e^{\alpha_i}\left( y_i \otimes 1 + 1 \otimes y_i - 1 \otimes 1\right)
$$
which follows from \eqref{E:deltaA}.
%
%Here, the operator $y_i$ acts on $K^*_T(\gr) \otimes_{R(T)} K^*_T(G/B)$ by 
%$y_i \bullet (\zeta \otimes \psi) =\zeta \otimes (y_i \bullet \psi)$ if $i \neq 0$ and $y_0 \bullet (\zeta \otimes \psi)$ is computed via \eqref{E:deltaA} and Theorem \ref{T:nilHecketheta}.
\section{Cohomology}\label{sec:cohom}
In this section, we indicate the modifications necessary for the preceding results to hold in cohomology.  

\begin{remark}
It would be interesting to deduce Theorem~\ref{T:H affine coproduct} (and other results in cohomology) directly from Theorem~\ref{T:affineschub} (and other $K$-theoretic results), for example by using the Chern character or by taking ``lowest degree" terms.  In particular, we do not know the relation between the coefficients $j^w_u$ in Theorem~\ref{thm:jzw} and the $k^w_u$ in Theorem~\ref{T:centralizerSchub}.
\end{remark}

\subsection{Small-torus affine nilHecke ring}
\label{S:HnilHecke}
Instead of $R(T)$, we work over $S = \Sym_\Z(\Xfin) \cong H^*_T(\pt)$.  The algebra $\bK$ is replaced by the small-torus affine nilHecke ring $\bA$, as defined in \cite[Chapter 4]{book}. Let $\bA_0$ be the nilCoxeter algebra, the ring generated by elements $A_i$ for $i\in \Iaf$ which satisfy the braid relations for $\Wa$
and the relation $A_i^2=0$. We have $\bA_0 = \bigoplus_{w\in \Wa} \Z A_w$ where
$A_w = A_{i_1} A_{i_2}\dotsm A_{i_\ell}$ for a reduced decomposition $w=s_{i_1}s_{i_2}\dotsm s_{i_\ell}$.
Let $Q(S)$ be the fraction field of $S$ and let $\bA_{Q(S)} = Q(S) \otimes_{\Q} \Q[\Wa]$ be the twisted group algebra of $\Wa$ with coefficients in $Q(S)$, with product $(q' \otimes w)(q \otimes v) = q' w(q) \otimes wv$ for $q,q'\in Q(S)$ and $w,v\in \Wa$. Then $\bA$ is the subring of $\bA_{Q(S)}$ generated by $S$ and $\bA_0$. We have the following analogue of \eqref{E:Tbasis}:
$
\bA = \bigoplus_{w \in \Wa} S A_w.
$
Instead of the Demazure product, we will make use of length-additive products.  Write $w \doteq uv$ if $w = uv$ and $\ell(w) = \ell(u) + \ell(v)$. Note that $A_u A_v=A_w$
if and only if $w\doteq uv$. This notation naturally extends to longer products.

\subsection{$\bA$-$\bA$-bimodule structure on cohomology of affine flag variety} 
Localization identifies $H^*_T(\flags)$ with a $S$-subalgebra of $\F(\Wa,S)$.  We identify a cohomology class $\xi \in H^*_T(\flags)$ with a function $\xi \in \F(\Wa,S)$ taking values $\xi(v)$, $v \in \Wa$.  For the small torus affine GKM condition see \cite[Section 4.2]{book}.

There is a $S$-bilinear perfect pairing $\langle \cdot,\cdot \rangle: \bA \times H^*_T(\flags)$ characterized by $\langle w, \xi \rangle = \xi(w)$.

There is a left action $\xi \mapsto a \cdot \xi$ of $\bA$ on $H^*_T(\flags)$ given by the formulae \cite[Chapter 4, Proposition 3.16]{book}
\begin{align}
(q \cdot \xi)(a) &= q\,\xi(a) \\
(A_i \cdot \xi)(a) &= A_i \cdot \xi(s_i a) + \xi(A_i a) \\
(w \cdot \xi)(a) &= w\,\xi(w^{-1} a)
\end{align}
for $a \in \bA$, $\xi \in H^*_T(\flags)$, $q \in S$, $i \in \Iaf$, and $w \in \Wa$. Here, $A_i$ acts on $S$ via
\begin{align}
\label{E:AonX}
	A_i(\la) &= \pair{\alpha_i^\vee}{\la} \id \\
\label{E:Aonproduct}
    A_i(q q') &= A_i(q)q' + (s_i \cdot q) A_i(q').
\end{align}

There is another left action $\xi \mapsto a \bullet \xi$ of $\bA$ on $H^*_T(\flags)$ given by \cite[Chapter 4, Section 3.3]{book}
\begin{align}
(a \bullet \xi)(b)  = \xi(ba)
\end{align}
for $a,b \in \bA$ and $\xi \in H^*_T(\flags)$.

Let $c_1(\cL_\la) \in H^*_T(\flags)$ denote the first Chern class of the $T$-equivariant line bundle with weight $\la \in X$ on $\flags$.  Explicitly \cite[Chapter 4, Section 3]{book}
\begin{align}\label{E:Hlinebundleloc}
	\pair{t_\mu v}{c_1(\cL_\la)} = v \cdot \la \qquad\text {$\mu\in Q^\vee$, $v\in \Wf$}
\end{align}

\begin{lem} \label{L:HXbullet} For any $\la\in X$ and $\xi\in H_T^*(\flags)$, we have $\la \bullet \xi = c_1(\cL_\la) \cup \xi$.
\end{lem}
%\begin{proof} Localizing at $t_\mu v$ for $\mu\in Q^\vee$ and $v\in\Wf$ we have
%\begin{align*}
%	\pair{t_\mu v}{\la\bullet \xi}
%	&= \pair{t_\mu v \la}{\xi} \\
%	&= \pair{\up{v}\la t_\mu v}{\xi} \\
%	&= \up{v}\la \pair{t_\mu v}{\xi} \\
%	&= \pair{t_\mu v}{c_1(\cL_\la)} \pair{t_\mu v}{\xi} \\
%	&= \pair{t_\mu v}{c_1(\cL_\la) \cup \xi}.
%\end{align*}
%\end{proof}
%
%
%\begin{rem}
%The action $\cdot$ coincides with Peterson's left action $\pi_L$ of $\bA$ on $H_T^*(\flags)$ and the action $\bullet$ is the opposite of Peterson's right action $\pi_R$ of $\bA$ on $H_T^*(\flags)$ \cite{Pet}.
%\end{rem}

\subsection{Endomorphisms}

Let $\PP = Z_{\bA}(S)$ be the centralizer of $S$ in $\bA$, called the Peterson subalgebra.  We have the cohomological wrong way map $\varphi^*:H^*_T(\flags) \to H^*_T(\gr)$.
%
%We have the following basic result.
%
%\begin{lem}\cite[Chapter 4, Lemma 4.8]{book} Let $F$ be the fraction field of $S$. Then $\PP = \bigoplus_{\mu \in Q^\vee} F t_\mu \cap \bA$.
%\end{lem}

\begin{thm}[{\cite{Pet} \cite{Lam:kschur} \cite[Chapter 4, Theorem 4.9]{book}}] \label{T:Hcentralizer}
\def\Htocentralizer{j}
There is an isomorphism $\Htocentralizer: H_*^T(\gr) \to \PP$ making the following commutative diagram of ring and left $R(T)$-module homomorphisms:
%, where $\tonilHecke$ sends the homology Schubert class $\psi_w$ to $T_w$ for $w\in \Wa$.
%, and $\varphi_*^{\bK}$ is the inclusion map.
\begin{align*}
\begin{diagram}
\node{H_*^T(\gr)}\arrow {s,t}{\varphi_*} \arrow {e,t}{\Htocentralizer} \node{\PP}  
\arrow {s,J}{}
\\ \node{H_*^T(\flags)} \arrow {e,b}{} \node{\bA} 
\end{diagram}
\end{align*}
\end{thm}

The maps 
\begin{align*}
p^*&:H_T^*(\thickaffgr)\to H_T^*(\flags) \\
\theta&: H^*_T(\flags) \to H^*_T(\flags) \\
\Delta&: H^*_T(\flags) \to H^*_T(\gr) \otimes_S H^*_T(\flags) \\
\ev_1^*&:H_T^*(G/B)\to H_T^*(\flags) \\
\eta&: H^*_T(\flags) \to H^*_T(\flags) \\
\kappa&: H^*_T(\flags) \to H^*_T(\flags)
\end{align*}
are defined as for $K$-theory.
Lemma \ref{L:wrongwayloc}, Proposition \ref{P:coprod}, Lemma \ref{L:coprod}, Lemma \ref{L:evoneloc} hold in cohomology with the obvious modifications.  Lemma \ref{L:evonelinebundle} holds with $c_1(\cL_\la)$ replacing $[\cL_\la]$.

\begin{prop}\label{P:coprodbullet}
For $\xi \in H^*_T(\flags)$ and $a \in \bA$, we have
$$
a \bullet \xi = \sum_{(\xi)} \xi_{(1)} \cup \eta(a \bullet \xi_{(2)})
$$
where $\Delta(\xi) = \sum_{(\xi)} \xi_{(1)} \otimes \xi_{(2)}$.
In particular, taking $a = 1$, we have the identity
$$
\cup \circ (1 \otimes \eta) \circ \Delta = 1
$$
in $\End_S(H^*_T(\flags))$.
\end{prop}

Lemmas \ref{L:kappaloc},\ref{L:endo}, and Propositions \ref{P:theta}, \ref{P:eta}, \ref{P:kappa} hold in cohomology. 

\subsection{Action of $\bA$ on tensor products}
Equation \eqref{E:deltaA} is replaced by
\begin{align}\label{E:HdeltaA}
	\Delta(A_i) &= A_i \otimes 1 + s_i \otimes A_i = 1 \otimes A_i + A_i \otimes s_i
\end{align}
Lemma \ref{L:cupproduct} holds with no change in cohomology.

\subsection{Tensor product decomposition of $H^*_T(\flags)$}
The left $\bA$-module structures via $\cdot$ on $H^*_T(\gr)$ and $H^*_T(G/B)$ give a left $\bA$-module structure on $H_T^*(\gr) \otimes_S H_T^*(G/B)$.
\begin{thm}\label{T:Htensor}
There is an $S$-algebra isomorphism
\begin{align}
\label{E:Htensormap}
  H_T^*(\gr) \otimes_S H_T^*(G/B) &\cong H_T^*(\flags) \\
\label{E:Htensormaponelements}
  a \otimes b& \mapsto p^*(a) \cup \ev_1^*(b)
\end{align}
with componentwise multiplication on the tensor product. This map is also an isomorphism of left $\bA$-modules under the $\cdot$ action.
\end{thm}

\subsection{Schubert bases}
The $S$-algebras $H^*_T(\flags)$, $H^*_T(\gr)$, and $H^*_T(G/B)$ have equivariant Schubert bases $\{\xi^x \mid x \in \Wa\}$, $\{\xi^u_{\gr} \mid u \in \Wz\}$, and $\{\xi^w_{G/B} \mid w \in \Wf\}$ respectively.  Equations \eqref{E:pschub} and \eqref{E:qschub} hold for cohomology Schubert classes.

The analogue of Proposition \ref{P:dotbullet} is as follows.

\begin{prop}\label{P:Hdotbullet} \cite[Chapter 4, Section 3.3]{book}
For $\la\in X\subset S$ and $A_i$ for $i\in\Iaf$, on the Schubert basis element $\xi^w \in H_T^*(\flags)$ for $w\in\Wa$, we have:
\begin{align}
	\label{E:HAleft}
	A_i \cdot \xi^w &= 
	\begin{cases}
		\xi^{s_iw} & \text {if $s_iw<s$} \\
		0 & \text {otherwise.}
	\end{cases} \\
	\label{E:HAright}
	A_i \bullet \xi^w &= 
	\begin{cases}
		\xi^{ws_i} & \text {if $ws_i<w$} \\
		0 & \text {otherwise.}
	\end{cases} \\
	\label{E:HSleft}
	\la \cdot \xi^w &= \la \xi^w \\
	\label{E:HSright}
	\la \bullet \xi^w &= c_1(\cL_\la) \cup \xi^w  .
\end{align}
\end{prop}

\subsection{Equivariant affine Stanley classes} 

\begin{thm}[\cite{Pet} \cite{Lam:kschur} \cite{book}] \label{thm:jzw}
For every $w\in \Wz$, $j_w = \tocentralizer(\xi_w^{\gr})$ is the unique element of $\PP$ of the form 
\begin{align}\label{E:Hjbasis}
j_w = A_w + \sum_{z\in \Wa\setminus\Wz} j^z_w A_z
\end{align}
for some $j^z_w\in S$.
\end{thm}

\begin{rem} By \cite{Pet} \cite{LS:quantumaffine} the coefficients $j^z_w$ are equivariant Gromov-Witten invariants for $G/B$.
\end{rem}

For $w\in\Wa$ the
equivariant affine Stanley class $F^w \in H_T^*(\gr)$ is defined by
\begin{align}\label{E:Haffstan}
	  F^w := \varphi^*(\xi^w)
\end{align}
and as usual we also consider $F^w$ an element of $H_T^*(\flags)$.
%Thus $p^*(F^w) = \theta(\xi^w)$.

\begin{lem}\label{L:HFcoeffs} For $w\in \Wa$, we have
\begin{align*}
	F^w &= \sum_{u\in \Wz} j^w_u \,\xi^u_{\gr}
\end{align*}
where the $j^w_u$ are defined in \eqref{E:Hjbasis}.
\end{lem}

\begin{prop}\label{P:Hcoprod1} For $w\in \Wa$, we have $\Delta(\xi^w) = \sum_{w \doteq w_1w_2} F^{w_1} \otimes \xi^{w_2}$.

\end{prop}

\subsection{Coproduct formula for affine Schubert classes}
\begin{thm} \label{T:H affine coproduct} 
For $w\in\Wa$, we have
	\begin{align*}
		\xi^w &= \sum_{\substack{(w_1,w_2)\in \Wa\times\Wf \\ w_1w_2 \doteq w}} F^{w_1} \cup \ev_1^*(\xi^{w_2}_{G/B}). %=\sum_{\substack{(w_1,w_2)\in \Wa\times\Wf \\ w_1w_2 \doteq w}} \theta(\xi^{w_1}) \cup \eta(\xi^{w_2}).
	\end{align*}
\end{thm}

%\begin{remark}
%Theorem~\ref{T:H affine coproduct} can be obtained from Theorem~\ref{T:affineschub} ``directly", as follows.  Begin with the dual statement \eqref{E:kuTv} $k_u T_v = \sum_{w \in \Wa} \sum_{\substack{x \in \Wa \\ w = x*v}}(-1)^{\ell(x)+\ell(v) - \ell(w)} k^{x}_u T_w$ inside $\bK$.  Multiplication in $\bA$ is obtained from multiplication in $\bK$ by replacing Demazure product with length-additive product.  
%
%We deduce that $j_u A_v = \sum_{w \in \Wa} \sum_{\substack{x \in \Wa \\ w = xv}} j^{x}_u A_w$ holds in $\bA$, where we have taken only the degree $\ell(u)+\ell(v)$ component, and used 
%\end{remark}

\subsection{Formulae for Schubert divisors}
\begin{cor} \label{C:Haffschubspecial} 
For $i\in\Iaf$ we have
\begin{align}\label{E:Haffschubreflection}
	\xi^{s_i} &= 
	\begin{cases}	
		F^{s_0} &\text {if $i=0$} \\
		F^{s_i} + \ev_1^*(\xi^{s_i}_{G/B}) & \text {otherwise.}
	\end{cases}
\end{align}
\end{cor}

\begin{prop}\label{P:Hreflectionclass} For all $i\in\Iaf$ and $\la\in Q^\vee$ we have
	\begin{align}
		\label{E:Freflection}
		F^{s_i} &= \level(\La_i) \, F^{s_0}.
	\end{align}
\end{prop}
%\begin{proof}
%Let $\{\xi^{x}_{\Taf} \mid x \in \Wa\}$ denote the equivariant %Schubert basis of $H_{\Taf}(\flags)$, where $\Taf \cong T \times %\C^\times$ denotes the affine torus.
%For all $i\in\Iaf$, in $H_{\Taf}(\flags)$ we have \cite{KK}
%\begin{align*}
%\xi^{s_i}_{\Taf}(w)&= w \cdot \La_i-\La_i\qquad\text {for all %$w\in\Wa$}
%\end{align*}
%%or equivalently
%%\begin{align}\label{E:divisorlinebundleaffine}
%%  \xi^{s_i}_{\Taf} =  c_1(\cL^{\flags,\Taf}_{\La_i})-\La_i.
%%\end{align}
%For all $\mu\in Q^\vee$ and $v\in\Wf$ we have
%\begin{align*} 
%	p^*(F^{s_i})(t_\mu v)
% =\xi^{s_i}(t_\mu) 
% = \cl(\xi^{s_i}_{\Taf}(t_\mu)) 
% = \cl(t_{\mu} \cdot \La_i-\La_i)).
%\end{align*}
%Applying this equation twice, we have
%\begin{align*}
%  &\quad\,\,p^*(F^{s_i})(t_\mu v) -\level(\La_i) p^*(F^{s_0})(t_\mu v)  \\
%  &= \cl(t_\mu \cdot %\La_i-\La_i)-\level(\La_i)\cl(\up{t_\mu}\La_0-\La_0) \\
%  &= \cl(t_\mu \cdot \af(\omega_i)-\af(\omega_i)) \\
%  &= 0. &\qedhere
%\end{align*}
%\end{proof}

\subsection{$\bullet$-action on affine Schubert classes}

Proposition \ref{P:nilHeckeeta} holds with $A_0$ replacing $y_0$ and $A_{-\theta} := -\theta^{-1}(1 - s_\theta) $ replacing $y_{-\theta}$.

%\fixit{Might want to omit these.}
%
%\begin{prop} \label{P:HnilHeckeeta} For $a\in\bA$ and $w\in \Wf$, we have
%\begin{align*}
%%	a \cdot \eta(\xi^w) &= \eta(\cl_{\bA}(a)\cdot \xi^w)  \\
%	a\bullet \eta(\xi^w) &= \eta(\cl_{\bA}(a)\bullet \xi^w)
%\end{align*}
%and in particular
%\begin{align*}
%%\label{E:s0bulleteta}
%	s_0 \bullet \eta(\xi^w) &= \eta(s_\theta\bullet \xi^w) \\
%%\label{E:A0bulleteta}
%	A_0 \bullet \eta(\xi^w) &= \eta(-A_\theta \bullet \xi^w)
%\end{align*}
%\end{prop}

\begin{thm}\label{T:HnilHecketheta}
For $\la\in X\subset S$, $i\in\Iaf$, and $u \in \Wz$, we have
\begin{enumerate}
\item
$\la \bullet \xi^u = c_1(\cL_\la) \cup \xi^u$,
\item
$A_i \bullet \xi^u =0$ and $s_i\bullet \xi^u=\xi^u$ if $i \in I$,
\item For $u\in\Wz\setminus\{\id\}$ 
\begin{align*}%\label{E:A0bulletSchub}
  A_0 \bullet \xi^u &= \xi^{us_0} 
  = \sum_{\substack{(x_1,x_2)\in\Wa\times\Wf \\ x_1x_2\doteq us_0 }}
	\theta(\xi^{x_1}) \cup \eta(\xi^{x_2}).
\end{align*}
\item For $u\in\Wz\setminus \{\id\}$
\begin{align*}%\label{E:s0bulletSchub}
	s_0\bullet \xi^u &= 
	\xi^u + \sum_{\substack{(x_1,x_2)\in\Wa\times\Wf \\ x_1x_2 \doteq us_0 }}
	\theta(\xi^{x_1}) \cup \eta(-\cl(\alpha_0)\bullet \xi^{x_2}).
\end{align*}
\end{enumerate}
\end{thm}

%
%We give the $\bullet$ action of a translation element.
%
%\begin{prop}\label{P:HtransbulletSchub} For $u\in\Wz$ and $\beta\in Q^\vee$ we have
%\begin{align*}\label{E:transbulletSchub}
%	t_\beta \bullet \xi^u = \sum_{\substack{(x,y,z)\in\Wa\times \Wf\times \Wz \\ xyz \doteq u}} \theta( \xi^x) \cup \eta(\pair{t_\beta}{\xi^z}\bullet \xi^y)
%\end{align*}
%\end{prop}

\begin{rem}
The $\cdot$ and $\bullet$ actions of $\bA$ make $H_T^*(\thickaffgr) \otimes_S H_T^*(G/B)$ into a left $(\bA\times\bA)$-module such that the map \eqref{E:Htensormap} is a left $(\bA\times\bA)$-module isomorphism.
\end{rem}

\subsection{Recursion}
The affine Schubert classes in the tensor product $H^*_T(\gr) \otimes_S H^*_T(G/B)$ are determined by the following recursion.
\begin{enumerate}
\item $\xi^u = \xi^u_\gr \otimes 1$ for $u \in \Wz$, and
\item For all $i\in\Iaf$ $$A_i \bullet \xi^w = 
	\begin{cases}
		\xi^{ws_i} & \text {if $ws_i<w$} \\
		0 & \text {otherwise.}
	\end{cases}$$
\end{enumerate}
Here, the operator $A_i$ acts on $H^*_T(\gr) \otimes_S H^*_T(G/B)$ by 
$A_i \bullet (\zeta \otimes \psi) =\zeta \otimes (A_i \bullet \psi)$ if $i \neq 0$ and $A_0 \bullet (\zeta \otimes \psi)$ is computed via \eqref{E:HdeltaA} and Theorem \ref{T:HnilHecketheta}.

\section{Examples}\label{sec:examples}

\subsection{Type $A$ in cohomology}
Letting $G = \SL(n)$, we now consider the affine Schubert polynomials \cite{Lee}.  Recall the isomorphism $H^*(\flags_G) \cong H^*(\gr_G) \otimes_{H^*(\pt)} H^*(G/B)$.  By \cite{Lam:kschur}, the cohomology $H^*(\gr_G)$ is isomorphic to $\Lambda/I_n$ where $\Lambda$ is the ring of symmetric functions and $I_n$ is the ideal $\langle m_\lambda \mid \lambda_1\geq n\rangle$ in $\Lambda$. Also, we have the classical Borel isomorphism $H^*(G/B)= \Z[x_1,\ldots,x_n] / \langle e_j(x_1,\ldots,x_n)\mid j\geq 1\rangle$ where $e_j(x_1,\ldots,x_n)$'s are elementary symmetric functions. Hence we have
$$H^*(\flags_G)\cong \Lambda/I_n \otimes_{\Z} \Z[x_1,\ldots,x_n] / \langle e_j(x_1,\ldots,x_n)\mid j\geq 1\rangle.$$ 
We list some affine Schubert polynomials for $n=3$, indexed by $w \in \tS_n$, the affine symmetric group.
\begin{center}
\begin{tabular}{|c||c|c|c|c|c|c|c|}
\hline
$w$ & 1 & $s_0$ & $s_1$ &$s_2$ &$s_1s_0$ &$s_2s_1$ & $s_2s_1s_0$ \\
\hline
$\asp_w$ & 1 & $h_1$ &$ h_1 + x_1$ & $h_1 + x_1 + x_2$ & $h_2$ & $h_2+h_1 x_1 + x_1^2$ & $m_{2,1}+m_{1,1,1} $\\
\hline
\end{tabular}
\end{center}
%
%\begin{align*}
%\asp_{id}&=1\\
%\asp_{s_0}&=h_1\\
%\asp_{s_1}&=h_1+x_1\\
%\asp_{s_2}&=h_1+x_1+x_2\\
%\asp_{s_1s_0}&=h_2\\
%\asp_{s_2s_1}&=h_2+ h_1 x_1 + x_1^2\\
%\asp_{s_2s_1s_0}&=m_{2,1}+m_{1,1,1}
%\end{align*}

The polynomial $\asp_{s_2s_1}$ can be computed in a number of different ways. First, we can start from $\asp_{s_2s_1s_0}$ which is the same as the affine Schur function indexed by $s_2s_1s_0$, and use the monomial expansion of the affine Schur functions \cite{Lam:affstan}. Then one can act with the divided difference operator $A_i \bullet$ to obtain $\asp_{s_2s_1}$. The action of $A_i \bullet$ is explicitly given in \cite[Definition 1.1]{Lee}. 

On the other hand, using the coproduct formula (Theorem~\ref{T:H affine coproduct}) directly give $\asp_{s_2s_1}$:
\begin{align*} \asp_{s_2s_1}&= F_{s_2s_1}+ F_{s_2} \mathfrak{S}_{s_1} + \mathfrak{S}_{s_2s_1}
= h_2+ h_1 x_1 + x_1^2
\end{align*}
where $F_w$ is the affine Stanley symmetric function,  the non-equivariant version of $F^w$ in Section 5, and $\mathfrak{S}_v(x)$ is the Schubert polynomial. Using the coproduct formula together with monomial expansions of $F_w$ \cite{Lam:affstan} and $\mathfrak{S}_v(x)$ \cite{BJS} provides the following theorem:

\begin{thm} \label{T:monomialpositive}Affine Schubert polynomials are monomial-positive.
\end{thm}

The same coproduct formulae hold in equivariant cohomology, with the affine double Stanley symmetric function $F^w$ \cite{LS:double} replacing $F_w$, and the double Schubert polynomial $\schub_v(x,y)$ \cite{LS:doubleSchub} replacing $\schub_v(x)$.  However, there is no combinatorially explicit formula for the equivariant affine Stanley classes $F^{w}$, see \cite[Remark 23]{LS:double}.

\subsection{Back stable limit}\label{ssec:backstable}
We explain how to obtain the coproduct formula \cite[Theorem 3.16]{LLS} for backstable Schubert polynomials from Theorem~\ref{T:H affine coproduct}.  Let $\bS_w(x) \in \Lambda \otimes_\Z \Z[x_i \mid i \in \Z]$ denote the back stable Schubert polynomial from \cite{LLS}, where $w \in S_\Z = \langle s_i \mid i \in \Z\rangle$ is an infinite permutation.  Let
$$
\phi_n:  \Lambda \otimes_\Z \Z[x_i \mid i \in \Z] \to  \Lambda/I_n \otimes_{\Z} \Z[x_1,\ldots,x_n] / \langle e_j(x_1,\ldots,x_n)\mid j\geq 1\rangle$$ 
denote the natural quotient ring homomorphism where in the second factor $x_i$ is send to $x_{i \mod n}$.  %Both rings have coproducts and $\phi_n$ commutes with the coproduct.  
The back stable Schubert polynomial has a unique expansion $\bS_w(x) = \sum_v a_v(x) \otimes \schub_v(x)$ where $a_v \in \Lambda$ and $\schub_v(x)$ is a finite Schubert polynomial, which may possibly involve negative letters.  By shifting $w$, we may assume that $v \in S_{>0} =  \langle s_i \mid i > 0\rangle$, so that $\schub_v(x)$ is a usual Schubert polynomial.  We will show that 
\begin{equation}\label{eq:bScoprod}
\bS_w(x) = \sum_{\substack{w \doteq uv \\ v \in S_{\neq 0}} } F_u(x) \otimes \schub_v(x)
\end{equation}
where $F_u(x)$ denotes the Stanley symmetric function and $S_{\neq 0} = \langle s_i \mid i \neq 0\rangle$.

For sufficiently large $n$, the permutation $w \in S_\Z$ gives a well-defined element of the affine symmetric group $\tS_n$, by sending $s_i$, $i \in \Z$ to $s_i$, $i \in \Z/n\Z$.  Abusing notation, we denote this element by $w \in \tS_n$ as well.  According to \cite[Theorem 10.9]{LLS}, for sufficiently large $n$, the image $\phi_n(\bS_w(x))$ is equal to the affine Schubert polynomial $\asp_w(x)$, and it is also known that the image $\phi_n(F_u(x))$ is equal to the affine Stanley symmetric function, also denoted $F_u$.

Given any nonzero $f(x) \in \Lambda \otimes_\Z \Z[x_i \mid i \in \Z]$, it is straightforward to see that for sufficiently large $n$, we must have $\phi_n(f(x)) \neq 0$.  Applying this to the difference of the two sides of \eqref{eq:bScoprod}, and using our Theorem~\ref{T:H affine coproduct}, we see that equality must hold in \eqref{eq:bScoprod}.

\subsection{Type $A$ in $K$-theory} 
Let $G = \SL(n)$.  We now consider affine versions of the Grothendieck polynomials.  We have the isomorphism $K^*(\flags_G) \cong K^*(\gr_G) \otimes_{K^*(\pt)} K^*(G/B)$ and identifications $K^*(G/B)= \Z[x_1,\ldots,x_n] / \langle e_j(x_1,\ldots,x_n)\mid j\geq 1\rangle$ \cite{LSS} and $K^*(\gr_G) \cong \widehat{\Lambda/I_n}$, where $\widehat{\Lambda/I_n}$ denotes the graded completion.  By Theorem \ref{T:affineschub}, we have the formula
\def\gp{{\mathfrak{G}}}
$$
\agp_w = \sum_{\substack{(w_1,w_2)\in \Wa\times\Wf \\ w_1*w_2=w}} (-1)^{\ell(w_1)+\ell(w_2) - \ell(w)} G^{w_1} \gp_{w_2} 
$$
in $\widehat{\Lambda/I_n} \otimes_\Z \Z[x_1,\ldots,x_n] / \langle e_j(x_1,\ldots,x_n)\mid j\geq 1\rangle$,
where $\agp_w$ is the affine Grothendieck polynomial, $G^{w_1} \in \widehat{\Lambda/I_n}$ denotes the affine stable Grothendieck polynomial \cite{LSS}, and $\gp_{w_2}$ is the Grothendieck polynomial of Lascoux and Sch\"utzenberger.  For example, let $n = 3$ and $w = s_2s_1$.   We have
\begin{align*}
\agp_{s_2 s_1} &= G^{s_2 s_1} - G^{s_2s_1} \gp_{s_1} + G^{s_2} \gp_{s_1} - G^{s_2} \gp_{s_2s_1} + \gp_{s_1s_1} \\
&=G^{s_2 s_1}(1 -  \gp_{s_1}) + G^{s_2}( \gp_{s_1} -  \gp_{s_2s_1}) + \gp_{s_2s_1}\\
&= G^{s_2 s_1}(1 -  x_1) + G^{s_2}(x_1-x_1^2) + x_1^2.
\end{align*}
From \cite[A.3.6]{LSS}, we have expansions in terms of Schur functions
\begin{align*}
G^{s_2} &= G_1^{(2)} = s_1 - s_{11} + s_{111} - s_{1111} + \cdots \\
G^{s_2s_1} &= G_2^{(2)} = s_2 - s_{21} + s_{211} - s_{2111} + \cdots 
\end{align*}
Note that the lowest degree term is $s_2+s_1 x_1+ x_1^2= \asp_{s_2s_1}$.
We plan to compare these formulae with the affine Grothendieck polynomials of Kashiwara and Shimozono \cite{KS} in future work.
%\Seungjin{Is there a positivity of affine Grothendieck up to obvious sign?}
%\Thomas{There is Lemma 7.11 in LSS. Is that what you mean?  It doesn't give an obvious sign for $\agp_w$, though?}

\subsection{Classical type in cohomology}
The affine coproduct formula in cohomology can be applied to obtain formulas for Schubert classes in finite-dimensional flag varieties $G/B$. For classical type we compare these formulas with the Schubert class formulas of Billey and Haiman \cite{BH} for nonequivariant cohomology and those of Ikeda, Mihalcea, and Naruse \cite{IMN} for equivariant cohomology, providing retrospective insight into these formulas.

The affine coproduct formula writes an affine flag variety Schubert class as a sum of products
of affine Grassmannian Schubert classes and $G/B$ Schubert classes.
The formulas of \cite{BH} and \cite{IMN} write a $G/B$ class of type $C_n$ or $D_n$ as a sum of products of cominuscule Grassmannian Schubert classes and type $A$ flag Schubert classes.  To compare our formulae with those in \cite{BH,IMN}, we use the fact that at the bottom of the affine Grassmannian of type $C_n^{(1)}$ or $D_n^{(1)}$ there is a copy of a cominuscule Grassmannian. To perform this comparison it is necessary to use an automorphism of the affine Dynkin diagram.

Consider the affine Dynkin diagrams of types $C_n^{(1)}$ and $D_n^{(1)}$ in Figure \ref{F:dynkin}.
\begin{figure}
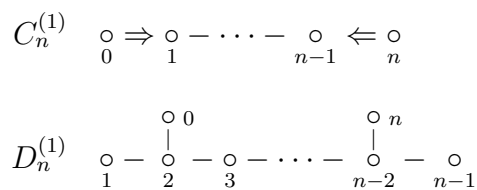

	\[
	\begin{array}{r@{\quad}l@{\qquad}l@{\quad}l}
	C_n^{(1)} & \dnode{}{0}\Rightarrow \dnode{}{1}-\cdots-\dnode{}{n-1}\Leftarrow\dnode{}{n} \\
	\\ 
	D_n^{(1)}  & \dnode{}{1}-\dnode{\ver{}{0}}{2}-\dnode{}{3}-\cdots-\dnode{\ver{}{n}}{n-2}-\dnode{}{n-1} \\
	\end{array}
	\]
	\caption{Affine Dynkin diagrams}
	\label{F:dynkin}
\end{figure}
Let $\tau$ be the affine Dynkin automorphism for type $C_n^{(1)}$ or $D_n^{(1)}$ given by $\tau(i)=n-i$ for $i\in\Iaf$.
There are two copies of the classical Weyl group $\Wf$ in $\Wa$: the usual one $W$ and $W'=W'_n=\tau(\Wf)$, which is generated by $s_j$ for $j\in \hat{I}\setminus\{n\}$
Let $G$ and $G'$ denote the subgroups of the corresponding loop group (or affine Kac-Moody group) with Weyl groups $W$ and $W'$ respectively, and let $G/B$ and $G'/B'$ be the two finite-dimensional flag varieties (either the symplectic flag variety or the orthogonal flag variety).  Finally, note that the subgroup of $\Wa$ generated by $s_j$ for $j\in \hat{I} \setminus\{0,n\}$ is isomorphic to the type $A_{n-1}$ Weyl group $W_{A_{n-1}}$.

For $w \in W'$, if we have $w \doteq w_1 w_2$ for $(w_1,w_2)\in \Wa \times W$, then $(w_1,w_2) \in W'\times W_{A_{n-1}}$.  Applying the affine coproduct formula (Theorem \ref{T:H affine coproduct}) and pulling back to $H^*_T(G'/B')$, we have in $H^*_T(G'/B')$ (with $G' = \Sp(2n)$ or $G' = \SO(2n)$) the equality, for $w \in W'$,
\begin{align}\label{E:classical coproduct}
	\xi'^w = \sum_{\substack{(w_1,w_2)\in W'\times W_{A_{n-1}} \\ w_1w_2\doteq w }} F^{w_1} \cup \tau^* \xi^{w_2}.
\end{align}
Here, $F^{w_1}$ is the pullback to $H^*_T(G'/B')$ (under the natural projection from the flag variety to a Grassmannian)
of an element of the torus-equivariant cohomology $H^*_T(\LG(n,2n))$ of the Lagrangian Grassmannian in the $C_n^{(1)}$ case, or an element of the torus-equivariant cohomology $H^*_T(\OG(n,2n))$ of the orthogonal Grassmannian in the $D_n^{(1)}$ case.  Also, $\xi^{w_2}$ denotes a Schubert class in $H^*_T(G/B)$ and $\tau^*$ is the composition of pullback maps $H^*_T(G/B) \to H^*_T(\flags_G) \to H^*_T(G'/B')$ (the first one being $\ev_1^*$).

%Now, there is a natural inclusion $\GL(n)/B \to G/B$, inducing a quotient map $H^*_T(G/B) \to H^*_T(\GL(n)/B)$.  (We write $\GL(n)$ instead of $\SL(n)$ as we are using the same $n$-dimensional maximal torus $T$.) Under this surjection, $\xi^w$ for $w \in W_{A_{n-1}}$ is mapped to the Schubert class $\xi^w_{W_{A_{n-1}}} \in H^*_T(\GL(n)/B)$ while $\xi^w$ for $w \notin W_{A_{n-1}}$ is sent to 0.  Any identity in the projective limit $\varprojlim_m H^*_T(\GL(m)/B)$ can thus be lifted to $H^*_T(G/B)$ for sufficiently large $n$.

Let us compare \eqref{E:classical coproduct} to the results of \cite{BH,IMN} following an argument similar to the one in Section~\ref{ssec:backstable}.  For concreteness, let us consider the 
Schubert polynomial $\schub_w^C \in \Z[z_1,z_2,\ldots] \otimes \Gamma$ of type $C$ \cite[Theorem 2.5]{BH} (type $D$ is similar), where $\Gamma \subset \Lambda$ is the subring spanned by $Q$-Schur functions (over $\Q$, the ring $\Gamma$ is generated by the odd power sum symmetric functions).  There is a ring homomorphism 
$$
\phi_n:   \Z[z_1,z_2,\ldots] \otimes \Gamma \to H^*(G'/B')
$$
taking $\schub_w^C$ to the Schubert class $\xi'^w$.  Now let $F_v^C \in \Gamma$ denote the type $C$ Stanley symmetric functions, which were studied in the classical setting in \cite{BH} \cite{La} \cite{La2}
and in the affine setting in \cite{LSS2} (see also \cite{Pon}).  By \cite{LSS2}, under $\phi_n$ we have that $F_v^C$ is sent to (the non-equivariant class) $F^v \in H^*(G'/B')$ and the usual Schubert polynomial $\schub_u(z)$ in $z$-variables is sent to $\tau^*(\xi^u)$.  According to \cite{BH}, the ring $\Z[z_1,z_2,\ldots] \otimes \Gamma$ injects into the projective limit $\varprojlim_n H^*_T(G'/B')$.  It follows from \eqref{E:classical coproduct} that we must have the expansion $\schub_w^C = \sum_{\substack{(w_1,w_2)\in W'\times W_{A_{n-1}}  \\ w_1w_2\doteq w }} F_{w_1}^C \schub_{w_2}(z)$, which is the Billey-Haiman formula for the type $C$ Schubert polynomials.

A similar formula holds in equivariant cohomology.  Equivariantly, our $F^{w_1}$ is a double analogue of the type $C$ or $D$ Stanley symmetric function.  Our \eqref{E:classical coproduct} gives a formula for the double Schubert polynomials of type $C$ or $D$ as a sum of products of double type $C$ or $D$ Stanley symmetric functions and type $A$ double Schubert polynomials.  Since the coproduct expansion of a Schubert class is unique, our formula must equal to that in \cite{IMN}.  Our definition of (equivariant) affine Stanley class gives a precise geometric description of the Grassmannian components of the formulas in \cite{IMN}.  We do not obtain a new proof of their formula, because we do not separately know that our $F^{w_1}$ can be compared to the combinatorics in \cite{IMN}.  See also \cite{AF,Tam}.  
%Finally, in the $n \to \infty$ limit, $\tau^* \xi^{w_2}$ can be identified with the usual Lascoux and Sch\"utzenberger double Schubert polynomial \cite{LS:doubleSchub}.

\subsection{Classical type in $K$-theory}
Our coproduct formula in equivariant $K$-theory should be compared with the classical type double Grothendieck polynomials of A. ~N.~Kirillov and H.~Naruse \cite{KN} \cite{HIMN} just as our cohomological formula relates to the work of Billey and Haiman.
There is a Pieri formula \cite{Tak} in the $K$-homology of the type $A$ affine Grassmannian, which gives some coproduct structure constants for $K$-cohomology Schubert classes.


\begin{thebibliography}{LLMSSZ}
	
%\bibitem[A]{A} Arabia.

\bibitem[AF]{AF} D.~Anderson and W.~Fulton. Degeneracy loci, Pfaffians, and vexillary signed permutations in types B, C, and D, preprint, 2012, {\tt arXiv:1210.2066}.

\bibitem[BJS]{BJS} S. Billey, W. Jockusch, and R. Stanley.
Some combinatorial properties of Schubert polynomials. 
J. Algebraic Combin. 2 (1993), no. 4, 345--374. 

\bibitem[BH]{BH} S. Billey and M. Haiman.  Schubert polynomials for the classical groups. J. Amer. Math. Soc. 8 (1995), no. 2, 443--482.

%\bibitem[G]{G} Ginzburg.

\bibitem[Gi]{Gi} V. Ginzburg. Geometric methods in the representation theory of Hecke algebras and quantum groups. Notes by Vladimir Baranovsky. NATO Adv. Sci. Inst. Ser. C Math. Phys. Sci., 514, Representation theories and algebraic geometry (Montreal, PQ, 1997), 127–-183, Kluwer Acad. Publ., Dordrecht, 1998.

\bibitem[HIMN]{HIMN} T.~Hudson, T.~Ikeda, T.~Matsumura, and H.~Naruse.
Double Grothendieck polynomials for symplectic and odd orthogonal Grassmannians. J. Algebra
546 (2020), 294--314.


\bibitem[IMN]{IMN} 
T. Ikeda, L. C. Mihalcea, and H. Naruse.
Double Schubert polynomials for the classical groups.
Adv. in Math. 226 (2011) no. 1, 840--886.

%\bibitem[Ka]{Ka} M. Kashiwara. The flag manifold of Kac-Moody Lie algebra, Algebraic analysis, geometry, and number theory (Baltimore, MD, 1988), 161--190, Johns Hopkins Univ. Press, Baltimore, MD, 1989.

\bibitem[Kac]{Kac} V. Kac. Infinite dimensional Lie algebras. Third edition. Cambridge University
Press, Cambridge, 1990.% xxii+400 pp. ISBN: 0-521-37215-1.

\bibitem[KK]{KK} B. Kostant and S. Kumar. T-equivariant K-theory of generalized flag varieties. J. Differential Geom. 32 (1990)  no. 2, 549--603.

\bibitem[KN]{KN} A. N. Kirillov and H. Naruse. Construction of double Grothendieck polynomials ofclassical types using idCoxeter algebras. Tokyo J. Math. 39, 3 (2017), 695--728.
%\bibitem[KP]{KP} V. Kac and D. Peterson.

%\bibitem[Kum]{Kum} S. Kumar. Kac-Moody groups, their flag varieties and representation theory, Progress in Mathematics, 204. Birkh\"auser Boston, Inc., Boston, MA, 2002.% xvi+606 pp.
%\bibitem[Ku2]{Ku:pos} S. Kumar, Positivity.
	
\bibitem[KS]{KS} M. Kashiwara and M. Shimozono.  Equivariant K-theory of affine flag manifolds and affine Grothendieck polynomials.  Duke Math. J. Volume 148, Number 3 (2009), 501--538.

\bibitem[Ku]{Ku} S. Kumar, Positivity in T-equivariant K-theory of flag varieties associated to Kac-Moody groups. With an appendix by Masaki Kashiwara. J. Eur. Math. Soc. (JEMS) 19 (2017), no. 8, 2469--2519.

\bibitem[La95]{La} T. K. Lam.
B and D analogues of stable Schubert polynomials and related insertion algorithms. Ph. D. Thesis, Massachusetts Institute of Technology. 1995.
 
\bibitem[La96]{La2} T. K. Lam.
$B_n$ Stanley symmetric functions. Proceedings of the 6th Conference on Formal Power Series and Algebraic Combinatorics (New Brunswick, NJ, 1994). Discrete Math. 157 (1996), no. 1--3, 241--270. 

\bibitem[Lam06]{Lam:affstan} T. Lam. Affine Stanley symmetric functions. Amer. J. Math. 128 (2006) no. 6, 1553--1586.

\bibitem[Lam08]{Lam:kschur} T.~Lam. Schubert polynomials for the affine Grassmannian. J. Amer. Math. Soc. 21 (2008), no. 1, 259--281.

\bibitem[LLS]{LLS} T.~Lam, S.-J.~Lee, and M.~Shimozono. Back stable Schubert calculus, preprint, 2018; {\tt arXiv:arXiv:1806.11233}.

%\bibitem[LLS]{LLS} S.-J.~Lee, Thomas Lam, and M. Shimozono,
%Polynomials for the Schubert basis of the equivariant $K$-theory of the affine flag manifold, in preparation.

%\bibitem[LLMS]{LLMS} T.~Lam, L.~Lapointe, J.~Morse, and M.~Shimozono.


\bibitem[LLMS]{LLMS} T. Lam, C. Li, L. C. Mihalcea, M. Shimozono.
A conjectural Peterson isomorphism in K-theory. J. Algebra 513 (2018), 326--343. 

\bibitem[LLMSSZ]{book} T.~Lam, L.~Lapointe, J.~Morse, A.~Schilling, M.~Shimozono, and M.~Zabrocki. k-Schur functions and affine Schubert calculus. Fields Institute Monographs, 33. Springer, New York; Fields Institute for Research in Mathematical Sciences, Toronto, ON, 2014. viii+219 pp.

\bibitem[LSSa]{LSS} T. Lam, A. Schilling and M. Shimozono. $K$-theory Schubert calculus of the affine Grassmannian. Comp. Math. 146 (2010), no. 4, 811--852.

\bibitem[LSSb]{LSS2} T.~Lam, A. Schilling and M. Shimozono.  Schubert polynomials for the affine Grassmannian of the symplectic group. 
Math. Z. 264 (2010), no. 4, 765--811.

\bibitem[LS]{LS:quantumaffine} T.~Lam and M.~Shimozono. Quantum cohomology of $G/P$ and homology of affine Grassmannian. Acta Math. 204 (2010), no. 1, 49--90.

\bibitem[LS2]{LS:double} T. Lam and M. Shimozono. 
$k$-double Schur functions and equivariant (co)homology of the affine Grassmannian.

%\bibitem[LS3]{LS:Pieri} T. Lam and M. Shimozono, Pieri rule for the equivariant homology of the affine Grassmannian.

\bibitem[LaSc]{LS:doubleSchub} A. Lascoux and M.-P.~Sch\"utzenberger.   Polyn\^omes de Schubert. C. R. Acad. Sci. Paris S\'er. I Math. 294 (1982), 447--450.

\bibitem[Lee]{Lee} S. J.~Lee.  
Combinatorial description of the cohomology of the affine flag variety. Trans. Amer. Math. Soc. 371 (2019), 4029--4057.

\bibitem[Mit]{M} S. A. Mitchell. The Bott filtration of a loop group. Algebraic topology, Barcelona, 1986, 215–-226, Lecture Notes in Math., 1298, Springer, Berlin, 1987. 

\bibitem[Pet]{Pet} D. Peterson. Quantum cohomology of $G/P$, Lecture notes, M.I.T., Spring 1997.

\bibitem[Pon]{Pon}
S. Pon. Affine Stanley symmetric functions for classical types. 
J. Algebraic Combin. 36 (2012), no. 4, 595--622.

\bibitem[Tak]{Tak} M.~Takigiku.
A Pieri-type formula and a factorization formula for sums of K-k-Schur functions, arXiv:1802.06335.

\bibitem[Tam]{Tam} H.~Tamvakis. Schubert polynomials and degeneracy locus formulas in ``Schubert Varieties, Equivariant Cohomology and Characteristic Classes", 261--314,
European Mathematical Society Series of Congress Reports, Z\"urich, 2018.

\end{thebibliography}
\end{document}